\documentclass{amsart}
 \usepackage{amssymb,epsfig}
\newcommand{\R}{\mathbf{R}}
\newcommand{\E}{\mathbf{E}}

\newcommand{\K}{\kappa}
\newcommand{\x}{X}

\theoremstyle{plain} 

\newcommand{\CAT}{\text{\rm CAT}}

\newcommand{\CBB}{\text{\rm CBB}}

\newcommand{\conv}{\mathbf{\breve{\mathcal{C}}}}
\newcommand{\conc}{\mathbf{\widehat{\mathcal{C}}}}

\newtheorem{thm}{Theorem}[section]
\newtheorem{cor}[thm]{Corollary}
\newtheorem{prop}[thm]{Proposition}
\newtheorem{lem}[thm]{Lemma}

\newtheorem{claim}{}

\theoremstyle{definition} 

\newtheorem{rem}[thm]{Remark}
\newtheorem{rems}[thm]{Remarks}
\newtheorem{defn}[thm]{Definition}

\DeclareMathOperator{\tang}{T}\DeclareMathOperator{\dirn}{\Sigma}
\DeclareMathOperator{\sn}{sn}
\DeclareMathOperator{\grad}{\nabla}
\DeclareMathOperator{\inte}{int}
\DeclareMathOperator{\cl}{closure\,}

\DeclareMathOperator{\diam}{diam}
\DeclareMathOperator{\id}{id}
\DeclareMathOperator{\length}{length}
\DeclareMathOperator{\speed}{speed}

\DeclareMathOperator{\diff}{d}

\DeclareMathOperator{\curv}{curv}
\DeclareMathOperator{\cone}{Cone\,}
\DeclareMathOperator{\dist}{\,dist}

\DeclareMathOperator{\dis}{\,distortion}

\numberwithin{equation}{section}
\begin{document}
\begin{abstract}
This paper completes a fundamental construction in Alexandrov geometry.
Previously we gave a  new construction of metric spaces with curvature bounds either above or below, namely warped products with intrinsic metric space base and fiber, and with possibly vanishing warping functions -- thereby extending the classical cone and suspension constructions  from interval base to arbitrary base, and furthermore encompassing gluing constructions.  This paper proves the converse, namely, all conditions of the theorems are necessary.    Note that in the cone construction, both the construction and its converse are widely used.  We also show that our  theorems for curvature bounded above and below, respectively, are dual. We give the first systematic development  of  basic properties of warped products of metric spaces with possibly vanishing warping functions, including new properties.
\end{abstract}

\title{Warped products admitting a curvature bound}
\author{Stephanie B. Alexander}\thanks
{This work was partially supported by a grant from the Simons Foundation (\#209053  Êto Stephanie
 Alexander).}
\address{1409 W. Green St., Urbana, Illinois 61801}
\email{sba@math.uiuc.edu}

\author{Richard L. Bishop}
\address{1409 W. Green St., Urbana, Illinois 61801}
\email{bishop@math.uiuc.edu}

\maketitle
\tableofcontents

\begin{flushright}
\today
\end{flushright}

\section{Introduction}\label{sec:intro}
This paper completes a fundamental construction in the theory of Alexandrov spaces. In classical Riemannian geometry, warped products provide perhaps the major source of examples and counter-examples.  In \cite{ab-wp}, a  new construction of metric spaces with curvature bounds either above or below was given, namely warped products with intrinsic metric space base and fiber, and with possibly vanishing warping functions -- thereby extending the classical cone and suspension constructions  from interval base to arbitrary base, and furthermore 
encompassing gluing 
constructions.

This paper proves  
the converse, 
namely, all the conditions of the theorems are necessary. 
Note that in the cone construction,   
the implications in both directions 
are widely used.  

Among the rather delicate arguments required, it turned out that  
adequate tools to handle all of them were not available when \cite{ab-wp} was written. 
For spaces of curvature bounded below, we know of no proof that does not use 
Petrunin's globalization theorem for incomplete spaces \cite{globalization}.  
We use it to prove a gluing theorem on the closure of the subset of the boundary on which  the warping function is nonvanishing. 
For  curvature bounded above, 
to prove the correct  bound on the fiber we had first to obtain a Gauss Lemma for curvature of general subspaces 
 \cite{AB-gauss}.  

In  \cite{ab-wp}, we did not recognize  that our  theorems for curvature bounded above and below, respectively, were completely dual.  Here we show this by proving equivalent formulations of the original hypotheses.

We also give a 
systematic development, including  new properties, of warped products of metric spaces with somewhere-vanishing warping functions.
These   were introduced in \cite{ab-wp}, where their treatment was  \emph{ad hoc}.
Somewhere-vanishing warping functions greatly enrich our source of examples and counter-examples by allowing  gluing on subsets.

The proofs given here illustrate a range of techniques and constructions in Alexandrov geometry.  We try to bring into focus the dualities between curvature bounded below and above

\section{Statement of theorems}
\label{sec:theorems}
Let $(\,X,\ |**\,|:X\times X\to [\,0,\infty)\,)$ be a metric space.  The \emph{model angle} $\tilde\angle^{\kappa} \bigl[{x^i}\!\!\!<^{\,x^j}_{\,x^k}\bigr]$ is the angle corresponding to $x^i$ in the \emph{model triangle} $\tilde\triangle^{\kappa}\bigl[x^i\,x^j\,x^k\bigr]$ with sidelengths $|x^i\,x^j|,|x^j\,x^k|,|x^k\,x^i|$, in 
the complete simply connected
surface of constant curvature $\kappa$.  We call that surface the \emph{model surface} for $\kappa$.  The model triangle and model angle are said to be \emph{defined} if there
is a unique triangle in the model surface with those sidelengths. In particular, the perimeter of the model triangle is $\le \varpi^\kappa=\pi/ \!\sqrt{\kappa}$ ($=\infty$ if $\kappa\le 0$).

As is well known, 
$X\in\CBB^\kappa$ and $X\in\CAT^\kappa$ may be defined using \emph{point-side $\kappa$-comparisons}. 
Namely, for every point $x^1$ and geodesic $[x^2\,x^3]$ such that  $\tilde\triangle^\kappa\bigl[x^1\,x^2\,x^3\bigr]$ is defined,  the  distance between $x^1$  and each point of $[x^2\,x^3]$  is $\,\ge\,$ (for $\CBB^\kappa$) or $\,\le\,$ (for $\CAT^\kappa$)
the distance between the corresponding points of $\tilde\triangle^\kappa
\bigl[x^1\,x^2\,x^3\bigr]$.

However, we are going to use instead, equivalent definitions that depend on distance only.  Our  $\CAT^\kappa$ definition is new in \cite{akp-book}.  See Section \ref{subsec:curv-bounds}.

Let  $f$ be a locally Lipschitz function 
defined on a metric space. For $\kappa\in\R$, we say $f$ is {\em sinusoidally}
 $\kappa$-{\em convex}, written 
 $f\in\conv^\kappa$, if 
 for every unit-speed geodesic $\gamma$,
 $$(f\circ\gamma)''+\kappa\cdot (f\circ\gamma)  \ge0.$$
If the inequality is reversed, 
we say $f$ is {\em sinusoidally} $\kappa$-{\em concave}, written  $f\in \conc^\kappa$.
The inequalities are meant in the generalized 
sense:   if $y''+\kappa\cdot  y=0$, and $y$ and 
$(f\circ\gamma)$ are defined on and coincide at the endpoints of a sufficiently short interval, then $(f\circ\gamma)\le y$  \,(respectively $\ge y$\,) That is, there are \emph{two-point supports}. Equivalently, at every point there is a solution $y$ of $y''+\kappa\cdot  y=0$ defined on an open interval, coinciding with $f\circ\gamma$ at that point, and satisfying the opposite inequality -- \emph{tangential supports} exist.
Thus sinusoidal $0$-convexity ($0$-concavity) is convexity (concavity) in the usual sense. 
\vspace{2mm}

\begin{defn}\label{defn:triple}
Let $B$ and $F$ be intrinsic spaces, and $f:B \to \R_{\ge 0}$\, be
locally Lipschitz.  Suppose $F\ne$ point, and $Z=f^{-1}(0)\neq B$.  
 When we denote a warped product by $B\times_fF$, we assume  $(B, f,F)$ is such a triple, which we call a \emph{WP-triple}. 
\end{defn}

\begin{thm} [$\CAT^\kappa$]\label{thm:mainA}\,\, 
Let $(B, f,F)$ be a WP-triple, 
and assume $f$ is Lipschitz on bounded sets. 
Then
$B \times_fF\in\CAT^\kappa$  if and only if the following conditions hold, where $Z=f^{-1}(0)$:\begin{enumerate}

\item  \label{mainA:base-f}
$B\in\CAT^\kappa$ and $f\in\conv^\kappa$.

\vspace{1mm}
\item 
\label{mainA:fiber-f>0}
If $Z = \emptyset$, then  $F\in\CAT^{\,\kappa_F}$
for 

\vspace{-4mm}
\[\kappa_F=\kappa\cdot (\inf f)^2.\]

\item
\label{mainA:fiber-f=0} 

If $Z \ne \emptyset$, then 
$F\in\CAT^{\,\kappa_F}$
for $\kappa_F=
\min\,\{\kappa_{\textrm{foot}}\,,\,\kappa_{\textrm{far}}\}$, where
\vspace{-4mm}

\begin{align*}
\kappa_{\textrm{foot}} =\inf\,\{(f\hspace{-.4mm}\circ\hspace{-.2mm}\alpha)^+\hspace{-.2mm}&\hspace{-.2mm}(0)^2 : \alpha =\!\dist_Z\hspace{-.5mm}\textrm{-realizer with footpoint } \alpha(0)\in Z,\, |\alpha^+(0)|=1\},\notag\\
&\kappa_{\textrm{far}}= \inf \{\kappa\cdot f(p)^2 : \ \dist_Z( p ) \ge \varpi^\kappa/2\}.
\end{align*}
\end{enumerate}
\end{thm}

\vspace{2mm}
\begin{thm} \label{thm:mainB} {\bf($\CBB^\kappa$)}\,\, 
Let $(B, f,F)$ be a WP-triple,
and assume $f$ is Lipschitz on bounded sets.
Then
$B\times_fF\in\CBB^\kappa$ if and only if the following conditions hold, where $Z=f^{-1}(0)$:

\begin{enumerate}
\item \label{mainB:base-f}
$B\in\CBB^\kappa$ and $f\in\conc^\kappa$.

\item \label{mainB:dag}
 Let \ $B^\dag (f)$\  be obtained by gluing two copies of
$B$ on $\ \cl(\partial B -Z)$,
and let  \ $f^\dagger:B^\dag\to[0,\infty)$\  be the tautological extension of $f$.  Then
$B^\dag(f)\in\CBB^\kappa$ and $f^\dag\in \conc^\kappa$.\vspace{1mm}

\item \label{mainB:fiber-f>0}   If $Z = \emptyset$, then 
$F\in\CBB^{\,\kappa_F}$ for

\vspace{-4mm}
\[\kappa_F=\kappa\cdot (\inf f)^2.\]

\vspace{1mm}
\item \label{mainB:fiber-f=0}
If $Z \ne \emptyset$, then  $F\in\CBB^{\,\kappa_F}$ for
\[\kappa_F = \sup\,\{(f\circ\alpha)^+(0)^2 : \alpha =\!\dist_Z\textrm{-realizer with footpoint } \alpha(0)\in Z,\, |\alpha^+(0)|=1\}.\
\]
 \end{enumerate}
\end{thm}
\vspace{2mm}

\begin{thm} \label{thm:grad}
\emph{(a)} In Theorem \ref{thm:mainA}\,(\ref{mainA:fiber-f=0}), we may substitute 
\[
\kappa_{\textrm{foot}} \,=\,
\liminf_{\epsilon \to 0}\ \{|\nabla_p(-f)|^2 : 
0<\dist_Z(p) \le \epsilon\}.
\]
\emph{(b)} In Theorem \ref{thm:mainB}\,(\ref{mainB:fiber-f=0}), we may substitute 
\[
\kappa_F \,=\,
\sup\,\{|{\grad_q}f\,|^2: \,q\in Z\}.\notag
\]
\end{thm}

\vspace{1mm}

\begin{rems}\label{rem:cone}
(a) In Theorem \ref{thm:mainA}, condition (\ref{mainA:base-f})  implies  $Z$ is $\varpi^\kappa$-convex. In Theorem \ref{thm:mainB}, condition (\ref{mainB:base-f}) implies  $Z\subset \partial B$.

\vspace{1mm}

\noindent(b) In Theorem \ref{thm:mainB} (\ref{mainB:fiber-f>0}) we may substitute
\[\kappa_F \ge \kappa \cdot f^2.\]
This is because when $Z = \emptyset$, conditions (\ref{mainA:base-f}) and (\ref{mainB:dag})
imply $\kappa \le 0$ (see proof of Lemma \ref{lem:Z-not-empty}).

\vspace{1mm}

\noindent(c)
Theorems \ref{thm:mainA}(2) and \ref{thm:mainB}(3) are asymptotic versions of a basic fact for warped products: if  $f$ achieves a positive minimum at $p_0$, then shorter joins between points of \,$\{p_0\} \times F$ 
cannot be achieved by leaving \,$\{p_0\} \times F$.

\vspace{1mm}

\noindent (d) \label{dir-space-curv}
 For the simple example  
 $$\cone F=\R_{\ge 0}\times_{\id}F,$$
Theorems \ref{thm:mainA} and \ref{thm:mainB} reduce to the well-known statements: $\cone F\in\CAT^0$  if and only if $ F\in\CAT^1$;  $\cone F\in\CBB^0$ if and only if $ F\in\CBB^1$. (See \cite[Theorem 4.7.1]{bbi}.) If we allow $F$ to be a disjoint union of intrinsic spaces $F_\alpha$, then in the first statement substitute: each component $F_\alpha\in \CAT^1$.  In the second statement, substitute: $F\in\CBB^1$ or $F= 2$ points. Here, $F= 2$ points is the only additional possibility since otherwise geodesic bifurcations occur at the vertex.
\end{rems}

\section{Background and conventions}
\label{sec:background}

Here we summarize our tools:

\subsection{Curvature bounds}\label{subsec:curv-bounds}
Definitions and basic theorems are discussed in \cite{bbi} and \cite[\emph{Definitions of CBB/CBA}]{akp-book};  and in 
 \cite{bgp}, \cite{pl}, \cite{shio} for $\CBB^\kappa$, and  \cite{bh} for $\CAT^\kappa$. 

Given a metric space $X$, we are going to use the following definitions.

\begin{itemize}
\item
A \emph{geodesic} $\gamma$ joining $x^1,x^2\in X$ is a constant-speed curve of length $|x^1x^2|$. 
We may also denote $\gamma$ by $[x^1\,x^2]$.  A \emph{pregeodesic} is a monotonically reparametrized geodesic. A geodesic (pregeodesic) is said to be unique if it is determined by its endpoints up to reparametrization. 
 $X$ 
is \emph{geodesic} (\emph{intrinsic}) if any $x^1,x^2\in X$  are joined by a geodesic (respectively, by curves of length arbitrarily close to $|x^1x^2|$). $X$ is  \emph{r-geodesic} (\emph{r-intrinsic}) if this condition is applied only when  $|x^1\,x^2|<r$.

\vspace{1mm}

\item
A quadruple of points $x^1,x^2,x^3,x^4$ in a metric space satisfies 
	\emph{$(1+3)$-point $\kappa$-comparison}, briefly \emph{$(1+3)^\K$},
if 
$$
\tilde\angle^{\kappa} \bigl[{x^1}\!\!\!<^{\,x^2}_{\,x^3}\bigr] 
+\tilde\angle^{\kappa} \bigl[{x^1}\!\!\!<^{\,x^3}_{\,x^4}\bigr]
+\tilde\angle^{\kappa} \bigl[{x^1}\!\!\!<^{\,x^4}_{\,x^2}\bigr]\le 2\cdot\pi, 
$$
or at least one of the three model angles $\tilde\angle^{\kappa} \bigl[{x^1}\!\!\!<^{x^i}_{x^j}\bigr]$ is undefined \cite{bgp}.

\vspace{1mm}

\item
A quadruple of points $x^1,x^2,x^3,x^4$ in a metric space 
satisfies
\emph{$(2+2)$-point $\kappa$-comparison}, briefly \emph{$(2+2)^\K$}, if
\begin{enumerate}
\item[(a)]
either $\tilde\angle^{\kappa} \bigl[{x^1}\!\!\!<^{\,x^3}_{\,x^4}\bigr] 
\le 
\tilde\angle^{\kappa} \bigl[{x^1}\!\!\!<^{\,x^3}_{\,x^2}\bigr]
+\tilde\angle^{\kappa} \bigl[{x^1}\!\!\!<^{\,x^2}_{\,x^4}\bigr]$, 
\item[(b)]
or $\tilde\angle^{\kappa} \bigl[{x^2}\!\!\!<^{\,x^3}_{\,x^4}\bigr]
\le \tilde\angle^{\kappa} \bigl[{x^2}\!\!\!<^{\,x^3}_{\,x^1}\bigr] 
+ \tilde\angle^{\kappa} \bigl[{x^2}\!\!\!<^{\,x^1}_{\,x^4}\bigr]$, 
\end{enumerate}
or at least one of the six model angles 
$\tilde\angle^{\kappa} \bigl[{x^i}\!\!\!<^{\,x^j}_{\,x^k}\bigr]$
is undefined \cite{akp-book}.
\end{itemize}

\vspace{1mm}
The following definitions of $\CBB^\kappa$ and $\CAT^\kappa$ are equivalent to point-side definitions, but 
 depend on distances only, not on existence of geodesics.  They allow us to give some  substantially simpler proofs.

\begin{defn}
 Let $X$ be an intrinsic space.  
 \begin{itemize}
 \item
\emph{$X\in\CBB^\kappa$} means $X$ is a
complete intrinsic space  in which every quadruple satisfies \,$(1+3)^\K$. In this paper, 
we  further assume  \emph{$X$ has finite dimension}.  
(In particular, $X$ is proper, hence a geodesic space, and boundary $\partial X$ is defined.) We also use the convention (for $\kappa >0$)
 that $X$ is not  isometric to
a closed interval of length $> \varpi^{\kappa}$, or  a circle of length $> 2\cdot\varpi^{\kappa}$.  Under this convention, if  $X\in\CBB^\kappa$, then 
$X$ has diameter $\le\varpi^\kappa$ \cite{bgp}.
 \item 
 $X$ has \emph{curvature $\ge\kappa$}, written \emph{$\curv X\ge \kappa$},  
if any point $p\in X$
has a neighborhood $\Omega_p$ such that all quadruples lying in $\Omega_p$ satisfy 
$(1+3)^\K$;    
\item
\cite{akp-book}
\emph{$X\in\CAT^\kappa$} means $X$ is a
complete intrinsic  space  in which every quadruple satisfies \,$(2+2)^\kappa$. It follows that $X$ is $\varpi^{\kappa}$-geodesic.
\item
$X$ has \emph{curvature $\le\kappa$}, written \emph{$\curv X\le\kappa$},  
if any point $p\in X$
has a neighborhood $\Omega_p$ such that all quadruples lying in $\Omega_p$ satisfy  $(2+2)^\kappa$.
\end{itemize}
\end{defn}

\subsection{Globalization}\label{subsec:globalization} 

The following theorem \cite{globalization} 
extends the Burago-Gromov-Perelman globalization theorem \cite{bgp} to incomplete spaces. 
Theorem \ref{thm:anton}   includes long intervals and circles in $\CBB^\kappa$, rather than using  our convention which excludes them.

\begin{thm}[Petrunin's incomplete-globalization theorem  \cite{globalization}
]\label{thm:anton} 
Let $X$ be a geodesic space and $ \hat X$ be its completion.
Suppose $\curv X\ge \kappa$.
Then $ \hat X\in\CBB^\kappa$.
\end{thm}

\subsection{Definitions}
Let $X$ be a metric space.
The \emph{speed} of a curve  $\alpha:J\to X$
at $t_0\in J$, where $J$ is an interval, is defined as
 $$(\speed\alpha)(t_0)= \lim_{t\to t_0,\,t\in J}(\,|\alpha(t)\,\alpha(t_0)|\,/\,|t-t_0|\,).$$ 
 If $\alpha$ is Lipschitz, then  \,$\speed\alpha$\, exists at almost all $t\in J$, and \,$\length \alpha$\,  is  finite and given by Lebesgue integral of  speed.

A subset $S\subset X$ will be called  \emph{convex in $X$} if all $x^1,x^2\in S$ are joined by  geodesics of $X$, and all such  geodesics lie in $S$. If this condition holds when $|x^1\,x^2|<r$, then $S$ is said to be \emph{$r$-convex}.

For $S\subset X$, we denote distance from $S$ by $\dist_S$. Set
\begin{equation}
\mathcal{B}(S,r)= \{x\in X: \dist_S(x) < r\}.
\end{equation}  
If $S=\{p\}$, we write $\mathcal{B}(p,r)$ for the open ball of radius $r$ about $p$, and 
$\overline{\mathcal{B}}(p,r)$ for the closure of $\mathcal{B}(p,r)$.

\subsection{Tangent spaces and differentials}
Suppose 
$X\in\CBB^\K$ or $X\in\CAT^\K$.  
Recall that if $\gamma^1$ and $\gamma^2$ are geodesics from $p$, and $x^i$ lies on $\gamma^i$, then $\tilde\angle^\kappa  \bigl[p\!\!\!<^{ x^1}_{x^2}\bigr]$ is a monotone function of  
$(|p \,x^1|, |p\, x^2|)$.  
Then the angle at $p$ between $\gamma^1$ and $\gamma^2$ is defined as 
\begin{equation}\label{eq:angle}
 \lim\,_{|p \,x^1|\to 0, \ |p \,x^2|\to 0}\ \ 
\tilde\angle^\kappa  \bigl[p\!\!<^{ x^1}_{ x^2}\bigr].
\end{equation}

Consider the set $\Gamma_pX$ of geodesics $\gamma$ with $\gamma(0)=p$.  Set $\gamma_1\sim\gamma_2$  if 
$\gamma_1$ and $\gamma_2$ 
are  non-constant and the angle between them is $0$. A metric on the quotient space $(\Gamma_pX/\sim)$ is given by  the angle between representative geodesics $\gamma$.  We denote this metric space by $\Sigma'_pX$, the \emph{space of geodesic directions}. 
The \emph{space of  directions} $\Sigma_pX$ is the completion of $\Sigma'_pX$.

The  \emph{tangent space}, or \emph{space of tangent vectors},  $T_pX$, is the linear cone over $\Sigma_pX$:
$$T_pX=\cone(\Sigma_pX).$$
If $\gamma$ is a geodesic  with $\gamma(0)=p$ and speed $c>0$, and $u$ is the direction at $p$ represented by $\gamma$, the \emph{right derivative $\gamma^+(0)$ of $\gamma$ at $0$} is the tangent ``vector'' $(c,u)\in T_pX$, which we write as $v=c\cdot u$. We denote the vertex of the cone $T_pX$ by $o_p$.

Let $f:X\to \R$ be a locally Lipschitz function such that $(f\circ\gamma)^+(0)$ exists for every geodesic $\gamma$ with $\gamma(0)=p$. 
Then the  \emph{differential of $f$ at $p$} is a uniquely determined, linearly homogeneous, Lipschitz map 
$${\diff}_pf: 
\tang_pX\to \R$$ 
such that  $({\diff}_pf)(x)=(f\circ\gamma)^+(0)$ when $\gamma$ is a geodesic with $\gamma^+(0)=x$.  

In this paper, convergence of spaces always refers to Gromov-Hausdorff convergence. We need:

\begin{lem}[\cite{bgp, Pm1} ]
\label{lem:blowup}
Suppose $X\in\CBB^\K$. 
Then
$$
(\,\tang_pX,\,o_p\,)=\lim_{\lambda\to\infty}(\lambda X,\,p\,),$$

\vspace{-4mm}

$$
(\,\tang_pX,\,\partial(\tang_pX),o_p\,)=\lim_{\lambda\to\infty}(\lambda X,\,\partial (\lambda X),p\,) \text{\ \  if \  }p\in \partial B.
$$
\end{lem}

\subsection{Convex functions and gradient vectors}\label{ssec:functions}
Suppose a continuous 
function $f$  on a metric space $X$ is sinusoidally $\kappa$-concave or $\kappa$-convex, i.e. $f\in\conc^\kappa$ or $f\in\conv^\kappa$.
Then  $f$ is   \emph{semiconcave} (\emph{semiconvex}\,), i.e. locally there is 
a constant generalized upper (\,lower\,) bound on $f''$ along unitspeed geodesics $\gamma$.  
Equivalently, $(f\circ\gamma)(t)-\lambda\cdot t^2$ is concave for some $\lambda\in\R$.
The restriction of a semiconcave function $f$ to a geodesic $\gamma$ has all the regularity properties
of a convex function: left and right derivatives exist at every point, and the second derivative exists almost everywhere.  

\begin{thm}\cite{ lytchak_open,petrunin_survey}
\label{thm:gradient}
Let $\x\in\CBB^\K$ or $\x\in\CAT^\K$, and 
$f:X \to \R_{\ge 0}$ be a locally Lipschitz semiconcave function. Then:
 \begin{itemize}
 \item[(i)]
${\diff}_pf$ exists  and is  concave.

\item[(ii)] 
The \emph{gradient} ${\grad}_pf\in T_pX$ exists, where ${\grad}_pf=o_p$ if ${\diff}_pf\le 0$, and otherwise
\[\label{eq;grad}
{\grad}_pf= ({\diff}_pf)(u_{\text{max}})\cdot u_{\text{max}}
\notag\]
for the unique $u_{\text{max}}\in\Sigma_p\x$ at which  $\,({\diff}_pf)\,|\,\Sigma_pX$ takes its maximum.

\item[(iii)]Maximal 
 \emph{gradient curves}, whose right tangents are everywhere equal to the gradient vector, exist and are unique.  
\end{itemize}
\end{thm}

\begin{proof}
Gradient curves of semiconcave functions were introduced in \cite{PP} (for $\curv\ge\K$), and their properties developed by Lytchak \cite{lytchak_open} (for both $\curv\ge\K$ and $\curv\le\K$) and Petrunin \cite{petrunin_survey}. In the $\CAT^\K$ case, existence of the gradient vectors and gradient curves as defined here follows from \cite{lytchak_open}  by invoking a Helly-type theorem (see \cite{lang-schroeder}). 
\end{proof}

\begin{rem}
When $X\in\CAT^\K$, we are going to apply Theorem \ref{thm:gradient} to semiconvex functions $f$, by considering the gradient vectors  and gradient curves of the semiconcave function $(-f)$.  We call the gradient vectors  ${\grad}_p(-f)$ the \emph{downward gradient vectors of $f$}, and the gradient curves of $-f$, the \emph{downward gradient curves of $f$}.
\end{rem}

\begin{rem}In \cite{ab-wp,ab-cbc}, sinusoidally $\kappa$-convex and $\kappa$-concave functions were called \emph{$\mathcal{F} \kappa$-convex} and \emph{$\mathcal{F} \kappa$-concave}.
\end{rem}

\section{Warped products}

 \label{sec:wp}
   
Basic properties of warped products with \emph{positive} warping functions, $f>0$,  were proved in \cite{had-wp}. They were used in \cite{ab-wp}, which treated vanishing of $f$ in an \emph{ad hoc} manner.  In this paper we require a
systematic treatment, including new properties, for $f\ge0$, 
given  in this section.  
 
Let $(B, f,F)$ be a WP-triple (Definition \ref{defn:triple}).

In this paper, 
$J$ always denotes some finite closed  interval.

\begin{defn}[Warped product]\label{def:wp} 
Consider the topological space $(B\times F)/\!\sim$, where the elements of $\{p\}\times F$ are identified if $f( p )=0$.  We denote this class by $\overline p$, or by  any of its representatives $(p,\varphi)$, $\varphi\in F$.

An \emph{admissible curve} for the triple $(B,f,F)$ is a curve $\gamma:J\to (B\times F)/\!\sim$.  We write $\gamma=(\gamma_B,\gamma_F)$ for $\gamma_B:J\to B$   
and $\gamma_F:J_+\to F$, where  $J_0=(f\circ\gamma_B)^{-1}(0)$, $J_+=J-J_0$.  
Set  $J_+=\sqcup_{i=1,2,\ldots} \,J_i$, where the $J_i$ are maximal open subintervals of $J_+$. 
We further assume $\gamma_B$ and
$\gamma_F|J_i$ to be Lipschitz, where   $\gamma_F|J_i$, $i=1,\ldots$, have a uniform Lipschitz constant. 
(Our class of admissible curves 
satisfies the concatenation property of  length structures in the sense of \cite{bbi};
 see remark at end of this section.)
Set 
\begin{equation}
\label{eq:length}
\length\gamma\  \ =\ \ \int _J\,\sqrt{
v_B^2+ (f\circ\gamma_B)^2\cdot v_F^2}\,,
\end{equation}
where $\int$ is Lebesgue integral, $v_B$ is the speed of $\gamma_B$, $v_F|J_+$ is  the speed of $\gamma_F|J_+$ and $v_F|J_0=0$.   
Then the integrand is defined  almost everywhere on $J$ and bounded.    

Equivalently,  
\[
\length\gamma\  \ =\ \  \sum_i \int _{J_i}\sqrt{
v_B^2+ (f\circ\gamma_B)^2\cdot v_F^2}  \ \ +\ \  \length\, (\gamma_B|J_0).
\]
   Here the first term
 is defined, 
 independently of enumeration, because the summands are positive.

Then the   \emph{warped product} \,$B\times_fF$\, is the corresponding intrinsic  space, where distance is given by infimum of lengths of admissible curves joining two given points. 

We refer to $B$ and $F$ as \emph{base} and \emph{fiber} respectively. 
$B\times\{\varphi_0\}$  is called a \emph{horizontal leaf};  and $\{p_0\}\times F$ when $f( p_0 )>0$,   a \emph{vertical leaf}.
\end{defn} 

\begin{rem}The vanishing set $f^{-1}(0)$ of $f$ represents the set on which the horizontal leaves $B\times\{\varphi_0\}$ are glued together.  At these points there is no well-defined projection $\gamma_F$ to $F$. 

\end{rem}  

\begin{prop}
\label{prop:wp-proj} 
The warped product $B\times_fF$ satisfies:
\begin{enumerate}

\item\label{horiz-leaf}
The intrinsic and extrinsic metrics   of any  horizontal leaf $B\times\{\varphi_0\}$ agree, and projection $(p,\varphi_0)\mapsto p$ is an isometry onto $B$.

\item\label{vert-leaf-proj}
If $f(p_0)\ne 0$, then the projection $(p_0,\varphi)\mapsto \varphi$ of any  vertical leaf \,$\{p_0\}\times F$, with its intrinsic metric,  is a homothety onto $F$ with multiplier $1/f(p_0)$.

\item\label{f-min}
If  $f$ achieves a positive minimum at $p_0$, then the intrinsic and extrinsic metrics of \,$\{p_0\} \times F$ agree.
\end{enumerate}
\end{prop}

\begin{proof}
Claims  (\ref{horiz-leaf}) and (\ref{vert-leaf-proj})  are immediate from the 
length formula (\ref{eq:length}).

Also by (\ref{eq:length}), the projection onto \,$\{p_0\} \times F$ given by  \,$(p,\varphi)\mapsto (p_0,\varphi)$ is length-nonincreasing  if $p_0$ is a minimum point of $f$.  
Hence  (\ref{f-min}).
\end{proof}

\begin{rem} 
A horizontal leaf need not be convex even if $B\times_fF$ is a geodesic space, since vanishing of the warping function$\,f$ allows geodesics to bifurcate into distinct horizontal leaves.
For instance, suppose $\alpha:[0,1]\to B$ is a geodesic of $B$ such that 
$f(\alpha(0))=f(\alpha(1))
=0$ 
and $f\circ\alpha$ is not identically $0$.
 Then for any distinct $\varphi_1, \varphi_2 \in F$, the geodesic $(\alpha,\varphi_2)$ of  $B \times_fF$ has its endpoints in $B\times\{\varphi_1\}$ but does not lie in $B\times\{\varphi_1\}$.
\end{rem}

Now we show that distance in a warped product is fiber-independent, in the sense that distances may be calculated by substituting for $F$ a different intrinsic  space.  
Propositions \ref{prop:wp-proj} and \ref{prop:fiber-independence} summarize properties that for the case $f>0$ are given in \cite{had-wp}.  Proposition \ref{prop:wp-proj}(\ref{geo-fiber-proj}) is due to Chen.

\begin{prop}[Fiber independence]\label{prop:fiber-independence}
Let $W=B\times_fF$ and $W^*=B\times_fF^*$, where 
$F^*\ne$ point\, is an intrinsic space.

\begin{enumerate}
\item\label{dist-fiber-indep}
Let $p,q\in B$, $\varphi,\psi\in F$, and $\varphi^*,\psi^*\in F^*$.
\[
\text{If  }\ \  \ 
|\varphi\,\psi |_F=|\varphi^*\,\psi^*|_{F^*}\,,\ \ \ \text{then }\ \ \
|(p,\varphi)\, (q,\psi) |_W=|(p,\varphi^*)\, (q,\psi^*)|_{W^*} 
.\]
\hspace*{-12mm}
Let $\gamma=(\gamma_B,\gamma_F):J\to W$ be a geodesic such that  $f\circ\gamma_B>0$: 
\vspace{1mm}

\item\label{geo-fiber-proj}  \cite{C}
$\gamma_F$ is a pregeodesic in $F$.

\item\label{geo-fiber-indep}
Suppose  $\beta^*:J\to F^*$ is a pregeodesic in $F^*$ such that  $\beta^*$ and $\gamma_F$ have the same  speed, i.e. $v_{F^*}=v_F$ where $v_{F^*}$ is the speed of $\beta^*$.  Then $(\gamma_B,\beta^*)$ is a geodesic in $B\times_fF^*$.

\end{enumerate}
\end{prop}
\begin{proof} 
Let $\gamma_i:J\to W$ be admissible curves with endpoints $(p,\varphi)$, $(q,\psi)$, where $\length\gamma_i\to |(p,\varphi)\, (q,\psi) |$.

Suppose $f\circ(\gamma_i)_B>0$.  Set $L_i=\length\, (\gamma_i)_F$ and $v_i=\speed\,(\gamma_i)_F$. Let $\beta^*_i$ be curves in $F^*$  with endpoints $\varphi^*$, $\psi^*$ and lengths $L_i^*\to |\varphi^*\,\psi^*|=|\varphi\,\psi |$. Without loss of generality, $L_i^*\le L_i+\epsilon_i$, $\epsilon_i\to 0$.  Define $\gamma_i^*:J\to W^*$ by setting $(\gamma_i^*)_B=(\gamma_i)_B$, and  letting \,$(\gamma_i^*)_{F^*}$\,  be \,$\beta_i^*$\, reparametrized with speed $(L_i^*/L_i)\cdot v_i$. By the length formula  (\ref{eq:length}),  $|(p,\varphi)\, (q,\psi) |=\lim\length(\gamma_i )\ge \limsup(\length\gamma_i^*)$. Hence  $|(p,\varphi)\, (q,\psi) |\ge |(p,\varphi^*)\, (q,\psi^*)|$.  

Suppose $(f\circ(\gamma_i)_B)^{-1}(0)\ne\emptyset$. By the length formula,  there is an admissible curve $\check\gamma_i$ with endpoints  $(p,\varphi)$, $(q,\psi)$ that is not longer than $\gamma_i$, such that $(\check\gamma_i)_B=(\gamma_i)_B$ and $(\check\gamma_i)_F$  is constant on each maximal subinterval on which  $f\circ(\gamma_i)_B>0$ . Thus we may assume  $(\gamma_i)_F$ has this form, hence $\length\gamma_i=\length(\gamma_i)_B$. There are  curves $\gamma_i^*$ of the same form  in $W^*$ with endpoints  
$(p,\varphi^*)$, $(q,\psi^*)$, such that $(\gamma_i^*)_B=(\gamma_i)_B$  and \,$\length\gamma_i^*=\length(\gamma_i)_B=\length\gamma_i$.  

Therefore in all cases, $|(p,\varphi)\, (q,\psi) |\ge |(p,\varphi^*)\, (q,\psi^*)|$.
Reversing the roles of $W$ and $W^*$ proves (1).

To prove (2), suppose $\gamma_F$ has length $L>|\varphi\,\psi|$. Set $v=\speed\gamma_F$. 
Let $\beta:J\to F$ be a curve with endpoints $\varphi$, $\psi$ and  length $L'<L$. Then the length of $\gamma$ is reduced  by replacing $\gamma_F$ with the reparametrization of $\beta$  with speed $(L'/L)\cdot v$.  This contradiction gives (2).

(3) is immediate from (1) and the length formula  (\ref{eq:length}).
\end{proof}

The two-piece property in the next proposition is  
worthy of note.

\begin{prop}[Vanishing warping function]
\label{prop:f=0} 
Let $\gamma=(\gamma_B,\gamma_F):J\to B\times_fF$ be a geodesic  joining $(p,\varphi)$ and $(q,\psi)$.  Suppose $J_0\ne\emptyset$ where $J_0=(f\circ\gamma_B)^{-1}(0)$.  
\begin{enumerate}

\item\label{f=0-fiber-projection}
 The  restriction of $\gamma_F$ to any maximal subinterval $J_i$ of $J-J_0$   is constant. 
 If $J_i$ Êhas no common endpoint with $J$, the constant can be changed to any other
Êpoint in $F$ and the resulting curve will still be a geodesic in $B\times_fF$ with the same endpoints.

\item\label{f=0-distance}
$|(p,\varphi)\,(q,\psi)| = \length \gamma_B$.

\item\label{f=0-base-projection}{\rm(Two-piece property)}
$\gamma_B$ consists of two geodesics of $B$ that intersect on the maximal subinterval $[t_0,t_1]$ of $J$ having endpoints in $J_0$. 

\end{enumerate}
\end{prop}

\begin{proof} 
By the length formula (\ref{eq:length}), any curve in $B$ joining $p$ and  $q$ and passing through $Z=f^{-1}(0)$ is the projection of a curve in $B\times_fF$ of the same length joining $(p,\varphi)$ and $(q,\psi)$, and such that the projection to $F$ on each interval $J_i$ is constant.  
Claims (\ref{f=0-fiber-projection}) and (\ref{f=0-distance}) follow.  

It follows also that  $\gamma_B$ minimizes length of curves in $B$ from $p$ to $q$ that intersect $Z$.  If $t_0=t_1$ in claim (\ref{f=0-base-projection}), it follows that 
$\gamma_B
 |[0,t_0]$ and 
 $\gamma_B |[t_0,1]$ 
 are geodesics and the claim holds. Suppose $t_0<t_1$.  Then 
$\gamma_B |[0,t_1]$ is  a geodesic, since otherwise  by the triangle inequality there is a curve from $p$ to $q$ passing through 
$\gamma_B(t_1)\in Z$ that is shorter than 
$\gamma_B$.  Similarly,  
$\gamma_B |[t_0,1]$ is a geodesic.
\end{proof}

Clairaut's theorem
on geodesics of a surface of revolution extends  to the metric setting.  The proof that the formulas hold almost everywhere (claim \ref{clairaut-ae} in the proof below) is  in  \cite{had-wp}.  Here we prove the new result that the speed $v_B$ exists and is continuous for all $t$; when $(f\circ\gamma_B)^{-1}(0)=\emptyset$, the same holds for $v_F$; and when $(f\circ\gamma_B)^{-1}(0) \ne \emptyset$,
setting ${\overline{v}}_F \equiv 0$ give a continuous extension of $v_F$.

\begin{thm}[Clairaut's theorem]
\label{thm:clairaut} 
Let $\gamma=(\gamma_B,\gamma_F):J\to B\times_fF$ be a geodesic with speed  $a$.  
Then $v_B$ and an extension $\overline{v}_F$
of $v_F$ are defined and
Lipschitz
 continuous for all $t\in J$, and 
 there is a constant $c(\gamma)$ such that
\begin{equation}\label{eq:clairaut}
(f\circ \gamma_B)^{2}\cdot \overline{v}_F=c(\gamma)\,; \end{equation}

\vspace{-6mm}

\begin{equation}\label{eq:clairaut1}
\text{if } f\circ\gamma_B>0,\,\, \text{ then }
v_B = \sqrt{a^2-(c(\gamma)/f\circ\gamma_B)^2}.
\end{equation}
\end{thm}

\begin{proof} 
Suppose $J_0\ne\emptyset$ where $J_0=(f\circ\gamma_B)^{-1}(0)$.  
By Proposition \ref{prop:f=0}\,(\ref{f=0-fiber-projection}),  $v_F:J-J_0\to F$ satisfies $v_F=0$. In this case, $v_F$ may be extended to all of $J$ by setting 
$\overline{v}_F \equiv 0$, and
(\ref{eq:clairaut}) holds with $c(\gamma)=0$. Moreover, from the two-piece property we conclude that $\gamma_B$ is a geodesic, with the only exception
possible being a single break point  when $t_0 = t_1$;  in this case, $v_B$ still exists   with constant value $a$.

So suppose $f\circ \gamma_B>0$.  

Since $\gamma_B$ and $\gamma_F$ are Lipschitz, the speeds $v_B$ and $v_F$ are  defined almost everywhere, and the Lebesgue integral of speed on an interval is arc-length.

\begin{claim}\label{clairaut-ae}
(\ref{eq:clairaut}) and (\ref{eq:clairaut1}) hold almost everywhere.  
\end{claim}

This claim is proved in \cite[Theorem 3.1]{had-wp}.   

\begin{claim} 
$v_F$ 
is defined and continuous on $J$ and satisfies (\ref{eq:clairaut}).
\end{claim}

Set $c=c(\gamma)$.  Since $c/(f\circ \gamma_B)^{2}$ is defined and continuous on $J$, $v_F$ has a continuous extension $\overline{v}_F$ to all of \,$J$\, by claim \ref{clairaut-ae}.  It follows that the arc-length function  $s(t)$ of   $\gamma_F$ is obtained by integrating the continuous function $\overline{v}_F$, and so $ds/dt=\overline{v}_F$.  Since $\gamma_F$ is a pregeodesic, $ds/dt$ is the speed of $\gamma_F$, i.e. $v_F =ds/dt=\overline{v}_F$.  The claim follows.

\begin{claim}
$v_B$ is defined and continuous on $J$ and satisfies (\ref{eq:clairaut1}).
\end{claim}

It suffices to assume $J$ is an open interval containing $0$, and show that $v_B$ is defined and continuous at $t=0$. 

Suppose $\speed\gamma =1$. Let  $f(\gamma_B(0))=b>0$.  For 
 $\epsilon>0$, consider a ball $\mathcal{B}=\mathcal{B}(\gamma_B(0),2r)\subset f^{-1}(b-\epsilon,b+\epsilon)$.  If $0<s_1,s_2< r$, then $\gamma_B|[-s_1,s_2]\subset \mathcal{B}$.
 
Comparing  the warped product metric with the
Cartesian product metric on
$B\times_{b+\epsilon}F$,\begin{equation}\label{eq:compare-product}
s_1+s_2 < \sqrt {|(\gamma_B(-s_1)\ \gamma_B(s_2)|^2
+ (b+\epsilon)^2\cdot |\gamma_F(-s_1)\ \gamma_F(s_2)|^2}.
\end{equation}
Since (\ref{eq:clairaut}) holds almost everywhere,
$$|\gamma_F(-s_1)\  \gamma_F(s_2)| = \int_{-s_1}^{s_2}\frac{c}{f(\gamma_B(s))^2}ds
< (s_1+s_2)\cdot\frac{c}{(b-\epsilon)^2}.$$
Hence by (\ref{eq:compare-product}),
$$|\gamma_B(-s_1)\ \gamma_B(s_2)|^2 >
(s_1+s_2)^2\cdot\left(1 -\frac{(b+\epsilon)^2\cdot c^2}{(b-\epsilon)^4}\right).$$

Similarly, comparison with  
$B\times_{b-\epsilon}F$ gives an upper bound:
$$|\gamma_B(-s_1) \  \gamma_B(s_2)|^2 <
(s_1+s_2)^2\cdot\left(1 -\frac{(b-\epsilon)^2\cdot c^2}{(b+\epsilon)^4}\right).$$

Dividing by $(s_1+s_2)^2$, and taking the limit first as $s_1, s_2 \to 0$, then as $\epsilon \to 0$, we obtain that $v_B^2$ exists and  equals $1- (c/b)^2$,  the value needed
for continuity.

By (\ref{eq:clairaut}), if  $\speed\gamma =a$ then $c(\gamma)=a\cdot c(\check\gamma)$ where  $\check\gamma$ is a unitspeed reparametrization of $\gamma$. Hence (\ref{eq:clairaut1}).

\begin{claim}
$v_B$ and $v_F$ are Lipschitz continous.
\end{claim}
This claim follows from (\ref{eq:clairaut}) and (\ref{eq:clairaut1}), since we assume $f$ is locally Lipschitz.
\end{proof}

\begin{rem}
The original formulation of (\ref{eq:clairaut1})  in \cite{had-wp} states that any geodesic for which $f$ is nonvanishing has a constant-speed reparametrization $\gamma$ satisfying 
$$
\frac{1}{2}\,{v_B}^2 + \frac{1}{2\,(f\circ \gamma_B)^2} = E.
$$
almost everywhere.
In this form, Clairaut's equation has a potential theoretic interpretation,  where the constant 
$E$ is called the total energy and the terms equated to $E$ are the
kinetic and potential energies.

\end{rem}

\begin{cor}[Vertical geodesics]\label{cor:vertical}
Let $\gamma=(\gamma_B,\gamma_F):[-s_0,s_0]\to B\times_fF$ be a geodesic with speed $a$, 
where   $\gamma(-s_0)$ and $\gamma(s_0)$ lie in the
same vertical leaf $\{p\}\times F$. Then
\begin{enumerate}
\item\label{F-midpt}
If $f\circ\gamma_B>0$, then $\gamma_F(0)$ is the midpoint of $\gamma_F$.

\item\label{symmetry}If $f\circ\gamma_B>0$, there is a geodesic $\check\gamma:[-s_0,s_0]\to B\times_fF$ with the same endpoints as $\gamma$ that is symmetric about its midpoint, i.e., $\check\gamma_B(-s) = \check\gamma_B(s)$.

 \item\label{F-midpt1}
 If $\,(f\circ\gamma_B)^{-1}(0)\ne\emptyset$, then $\gamma_B(0)$ is a nearest point to $p$ of $Z=f^{-1}(0)$.

 \item\label{min}
The minimum value of  \,$f\circ\gamma_B$\, is \,$(f\circ \gamma_B)(0))=c(\gamma)/ a$.

\item\label{speed-0}
If $f\circ\gamma_B>0$, then the speed $v_B$ of $\gamma_B$ satisfies 
$v_B(s)=0$ if and only if  $(f\circ\gamma_B)(s)=c(\gamma)/ a$.

\end{enumerate} 
\end{cor}

\begin{proof}
Suppose  $f\circ\gamma_B >0$.  
 
It suffices to take 
 $F=[-\ell/2,\ell/2]$, $\gamma_F(-s_0)= -\ell/2$,  $\gamma_F(s_0)=\ell/2$, 
where $\ell=\length\gamma_F$. Indeed, by Proposition \ref{prop:fiber-independence}\,(\ref{geo-fiber-indep}),
 $\gamma_B$ remains unchanged by this substitution; and by  (\ref{eq:clairaut}), $c(\gamma)$ also remains unchanged.

Without loss of generality, the midpoint of  $\gamma_F$ is $\gamma_F(s_1)$ for $s_1\in [0,s_0)$.

The reflection of $F$ in $0$ induces an isometric reflection of $B \times_f F$. Let $\check\gamma: [-s_0+2s_1,s_0]\to B\times_fF$ be defined in two halves, for which the second half coincides with $\gamma$ on $[s_1,s_0]$ and first half traces in reverse the reflection of the first half on $[-s_0+2s_1, s_1]$. Then  $\check\gamma$ has the same endpoints as $\gamma$; both arcs have the same length, $a(s_0-s_1)$, hence $\length\check\gamma =  
2a(s_0-s_1)\le 2as_0= \length \gamma$, with equality if and only if $s_1=0$.  Since $\gamma$ is a geodesic, $s_1=0$, hence  (\ref{F-midpt}) and (\ref{symmetry}).

Since $\check v_B(0)=0$ where $\check v_B$ is the speed of $\check\gamma_B$, we also have $v_B(0)=0$.  By (\ref{eq:clairaut1}), the vanishing set of $v_B$ is also the set on which  $f\circ\gamma_B$ takes its minimum value $c(\gamma)/a$, hence (\ref{min}) and (\ref{speed-0}).

Alternatively, suppose $(f\circ\gamma_B)^{-1}(0)\ne \emptyset$. 
Since $\gamma_B$ minimizes length of loops in $B$ at $p$  that intersect $Z$, we have (\ref{F-midpt1}).
(\ref{min}) is immediate since $c(\gamma)=0$. 
\end{proof}

\begin{rem}
Let $L$ denote  \,$\length\gamma$\, as defined by (\ref{eq:length}).  Then \,$L=L_{\,\Sigma}$,\, for\begin{equation}
\label{eq:length1}
L_{\,\Sigma}\,\,=\,\sup_{t_0<\ldots<t_n}\,\sum_{i=1}^n \,d_{\,i}\,,
\end{equation}
where the supremum is taken over all partitions  \,$t_0<\ldots<t_n$\, of $J$.  Here, letting $\overline t_i$ be a minimum point of $(f\circ\gamma_B)|[t_{i-1}\,t_i]$\,, 
\[
d_i=
\begin{cases}
\bigl |\gamma(t_i)\,\,\gamma(t_{i-1})\bigr |_{\,B\,\times\,(f\circ\gamma_B)(\overline t_i)\,\cdot F}\,
& \text{if }(f\circ\gamma_B)(\overline t_i)>0,\\
\bigl |\gamma_B(t_i)\,\,\gamma_B(t_{i-1})\bigr |_B
\quad & \text{if } (f\circ\gamma_B)(\overline t_i)=0,
\end{cases}
 \]
where $B\,\times\,(f\circ\gamma_B)(\overline t_i)\cdot F$\, denotes the Cartesian product  of $B$ with a scaling of $F$.
 The choice of  $\overline t_i$ ensures the formula is well-defined and any sequence of successively refined   sums is nondecreasing. 
The proof that \,$L=L_{\Sigma}$\,  proceeds  as in the classical case  for length of an absolutely continuous curve (\cite{g}).
 
Since we are assuming $f$ is locally Lipschitz, it can be verified that 
the length induced by the warped product metric agrees with $L$.  

 It turns out that mere continuity of $f$ is not sufficient
 for the warped product  construction, 
 even if $f>0$.  
(This corrects a misstatement in \cite{had-wp,ab-wp}.)
For example, let $B$ be the union of intervals $B_i=[0,1/i]$ glued at $0$, $f(0)= 1$, $(f|B_i)(1/i)=1/i$, $F=[0,1]$. For  $\gamma(t) = (0,t)$, $ 0 \le t \le 1$,  
we have $L=1$, while the infimum of lengths of admissible curves joining any two points of $\gamma$ is $0$.

When $f$  is positive and locally Lipschitz,  the original  and metric topologies agree, and  Definition \ref{def:wp} determines a length structure in the sense of  \cite{bbi}. 
 
\end{rem}
\section{Base and warping function, $\CAT$
}
Let $y=\sn^\kappa$ be the function on $\R$ satisfying
$$y''+\kappa y=0,\quad y(0)=0,\quad y'(0)=1.$$

\begin{thm}
\label{thm:wf-convexity}
\begin{enumerate}
\item \label{d-convex-set}
Suppose $X\in\CAT^\kappa$  and $S\subset X$ is 
$\varpi^{\kappa}$-convex. Then
$$\sn^\kappa\circ \dist_{S}$$
\noindent is a sinusoidally $\kappa$-convex function on $\,\mathcal{B}(S,\varpi^{\kappa}/2)$.

\vspace{2mm}

\item \cite{Pm1}\label{d-bdry}
Suppose  $X\in\CBB^\kappa$ and $\partial X\ne\emptyset$.  Then
$$\sn^\kappa\circ \dist_{\partial X}$$
is a sinusoidally $\kappa$-concave function on $X$.
\end{enumerate}
\end{thm}

\begin{proof} See  \cite[\S 3]{ab-cbc}, or for  (\ref{d-bdry}), see  \cite[Theorem 3.3.1]{petrunin_survey}.
\end{proof}

The next lemma will allow us often to restrict attention to warped products whose fibers are intervals.

\begin{lem} \label{lem:F=interval} 
Let $(B, f,F)$ be a WP-triple. Suppose $B \times_fF\in\CAT^\kappa$ or $B \times_fF\in\CBB^\kappa$ respectively.

\begin{enumerate}
\item\label{F=int:2-parameter}
Let $\beta: J\to F$ be a unit-speed  geodesic.  Then under the embedding
$$\id\times\beta: B\times_fJ\to  B \times_fF,$$  the intrinsic and extrinsic metrics of $B\times_fJ$ agree.

\item\label{F=int} 
There is a nontrivial  interval $J$ such that $B \times_fJ\in\CAT^\kappa$ or $B \times_fJ\in\CBB^\kappa$ respectively.
\end{enumerate}
\end{lem}

\begin{proof}
By Proposition \ref{prop:fiber-independence}\,(\ref{geo-fiber-indep}) and Proposition \ref{prop:f=0}, the map 
$\id\times\beta$ preserves geodesics.  Therefore (\ref{F=int:2-parameter}) holds. 

By Proposition   \ref{prop:fiber-independence}\,(\ref{geo-fiber-proj}), $F$ is $\varpi^\kappa$-geodesic if $B \times_fF\in\CAT^\kappa$, and $F$ is geodesic if  $B \times_fF\in\CBB^\kappa$.
Therefore (\ref{F=int}) follows from (\ref{F=int:2-parameter}) and the assumption $F\ne$ point.
 \end{proof}

\begin{lem}\label{lem:f-lim}
If $W=B\times_f J$ is a geodesic space, where $J$ is an interval with interior point $0$, then the warping function $f:B\to\R_{\ge 0}$ satisfies
$$f(p) = \lim_{\epsilon\to 0}\frac{(\sn^\kappa \circ  \dist_{B\times\{0\}})((p,\epsilon))}{\epsilon}.$$
\end{lem}

\begin{proof}
By Proposition \ref{prop:wp-proj}\,(\ref{vert-leaf-proj}),  
for $p\in B$ the  curve $\beta(t)= (p,t)$ in $W$  is rectifiable, with arc-length  parameter $s$ satisfying $t=f(p)s$.  Then the speed of $\beta$ exists and equals $f(p)$ almost everywhere.  
By Proposition \ref{prop:fiber-independence}\,(\ref{dist-fiber-indep}) (taking $F=F^*=J$ and $p=q$),  if
 two points of  $\{p\} \times J$ have the same intrinsic distance in $\{p\} \times J$ then they have the same extrinsic distance in $W$.  Thus the speed of $\beta$  is constantly $f(p)$ everywhere, and we have 
\begin{align*}
f(p) 
=\lim_{\epsilon\to 0} \frac{|(p,\epsilon)\ (p,-\epsilon)|}{2\epsilon}.
\end{align*}

A geodesic $\gamma$  realizing the distance from  $(p, \epsilon)$
to $B \times \{0\}$ has a symmetric extension, 
which is a geodesic  between $(p, \epsilon)$ and $(p, -\epsilon)$ since $\gamma$ cannot be shortened.  Thus
\begin{align}\label{eq:flimit}
f(p) 
= \lim_{\epsilon\to 0}\frac{ \ \dist_{B\times\{0\}}((p,\epsilon))}{\epsilon}.
\end{align}
Since ${\sn^\kappa}'(0) = 1$, the lemma follows.
\end{proof}

\begin{thm}[Theorem \ref{thm:mainA}\,(\ref{mainA:base-f})]
\label{thm:A:base-f} 
Let $(B, f,F)$ be a WP-triple.  Set $Z=f^{-1}(0)$. If $B \times_fF\in\CAT^\kappa$, then $B\in\CAT^\kappa$ and $f\in\conv^\kappa$.  
\end{thm}

\setcounter{claim}{0}

\begin{proof}
Proposition \ref{prop:wp-proj} (\ref{horiz-leaf}) implies 
$B\in\CAT^\kappa$.

By Lemma \ref{lem:F=interval}\,(\ref{F=int}),  we may assume $F$ is a non-trivial interval $J=[-\theta_0,\theta_0]$. 
Since any two points at distance $<\varpi^{\kappa}$ in  $B \times_fJ$ are joined by a  unique geodesic, Proposition \ref{prop:wp-proj}\,(\ref{horiz-leaf}) 
 implies that each horizontal leaf  $B\times\{\epsilon\}$ is a $\varpi^{\kappa}$-convex subset of $B \times_fJ$.  
By Theorem \ref{thm:wf-convexity}\,(\ref{d-convex-set}),  $\sn^\kappa\circ\, \dist_{B\times\{0\}}$ is  sinusoidally $\kappa$-convex on  the tubular neighborhood \,$\mathcal{B}\left(B\times\{0\},\,\varpi^{\kappa}/2\right)$, and hence on a neighborhood of $(p,\epsilon)$ in $B\times\{\epsilon\}$ for $\epsilon$ sufficiently small.
By  Lemma \ref{lem:f-lim} and Proposition \ref{prop:wp-proj}\,(\ref{horiz-leaf}),  $f\in\conv^\kappa$. 
 \end{proof}

\section{Base and warping function, $\CBB$
} 
\label{sec:dagger}

Recall that  we write $\overline p\in B\times_fF$ when $f(p)=0$, where $\overline p$  is the equivalence class $\{(p,\varphi):\varphi\in F\}$.

The  next lemma contains what we need in this paper about tangent cones of warped products.

\begin{lem}\label{lem:tan-cone}
Let $(B, f,F)$ be a WP-triple, and $J$ be a  closed interval.  
\begin{enumerate}

\item
\label{mainB-base}
Suppose $B \times_f F\in \CBB^\kappa$.    Then 
$B\in\CBB^\kappa$.

\item
\label{diff}
Suppose $B \times_f F\in \CBB^\kappa$.  If $f( p ) =0$, then ${\diff}_pf$ is defined  and 
 $$\tang_{\overline p}(B\,\times_fF)\  =\  \tang_p\!B\times_{{\diff}_pf}F,$$
$$\dirn_{\overline p}(B\times_fF) \ = \ \dirn_p\!B\times_{({\diff}_pf|\dirn_pB)}F.
$$

\item\label{f>0-tan} 
Suppose $B \times_f J\in \CBB^\kappa$. If $f(p) > 0$, then 
$$\tang_{(p,\varphi)}(B\times_fJ) \ = \ 
\begin{cases}
\tang_p\!B\times \R_{\ge 0},\quad & \text{if $\varphi =$ endpoint of $J$},\\
\tang_p\!B\times \R,\quad & \text{if $\varphi$ = interior point of $J$}.
\end{cases}
$$
$$\Sigma_{(p,\varphi)}(B\times_fJ)=
\begin{cases}
[0,\pi/2]\times_{\sin\circ\id}\Sigma_pB,\quad & \text{if $\varphi =$ endpoint of $J$},\\
[0,\pi]\times_{\sin\circ\id}\Sigma_pB,\quad & \text{if $\varphi$ = interior point of $J$}.
\end{cases}
$$
\end{enumerate}
\end{lem} 

\begin{proof}

(\ref{mainB-base}) follows from Proposition \ref{prop:wp-proj}\,(\ref{horiz-leaf}).   
When $\dim B=1$, we additionally use  $\diam (B \times_f F)\le\varpi^\kappa$.

(\ref{diff}) By the arc-length formula (\ref{eq:length}),$$\lambda(B\times_fF) = (\lambda B)\times_{ \lambda(f\circ i_\lambda)}F,
$$
where $i_\lambda:\lambda X\to X$ is the tautological map.

By Lemma \ref{lem:blowup}, 
$$
\bigl(\tang_{\overline p}(B\times_fF),o_{\overline p}\bigr) =\lim_{\lambda\to\infty} \bigl((\lambda B)\times_{ \lambda(f\circ i_\lambda)}F ,\,\overline p\,\bigr).
$$
The existence of this limit implies that $(f\circ\alpha)^+(0)$ exists for every geodesic $\alpha$ of $B$ with $\alpha(0)=p$. (\ref{diff}) follows.

(\ref{f>0-tan}) We may also write
$$\lambda(B\times_fF) = (\lambda B)\times_{f\circ i_\lambda}( \lambda F).
$$
Thus we obtain (\ref{f>0-tan}), e.g.  when  $f(p) >  0$ and $\varphi $ is an interior point of $J$,
$$T_{(p,\varphi)}(B\times_fJ) =\lim_{\lambda\to\infty} \bigl((\lambda B)\times_{f\circ i_\lambda}( \lambda J) ,\,(p,\varphi)\,\bigr))=T_pB \times (f( p )\R)= T_pB \times \R.$$
\end{proof}

\begin{thm}[Theorem \ref{thm:mainB}\,(\ref{mainB:base-f})]\label{thm:mainB:base-f}
Let $(B, f,F)$ be a WP-triple. Suppose $B \times_f F\in \CBB^\kappa$. Then for $Z=f^{-1}(0)$:
 \begin{enumerate}
 \item
$B\in\CBB^\kappa$,

\item\label{Z-in-bdry}
$Z\subset\partial B$,
 
\item\label{mainB-f}
$f\in\conc^\kappa$.
\end{enumerate}
\end{thm}

\begin{proof}
By Lemma \ref{lem:F=interval}.(\ref{F=int}), we may assume $F$ is a non-trivial interval $J=[-\theta_0,\theta_0]$. 

 \begin{claim}
$B\in\CBB^\kappa$.
\end{claim}
See Lemma  \ref{lem:tan-cone}\,(\ref{mainB-base}).

 \begin{claim}\label{cl:f>0-concave}
$f\,|\,\bigl(B-(Z\cup\partial B)\bigr)\,\in\,\conc^\kappa$.
\end{claim}
Set $W=B\times_f[0,\theta_0]$.

For any curve in $B \times_f J$ connecting two points of $W$, any maximal open segment not in $W$ can be reflected into $W$, thus giving a curve of equal length in $W$. Hence intrinsic distance in $W$   equals distance in  $B\times_fJ$.  Therefore $W\in \CBB^{\,\kappa}$.

Consider $p\in B-(Z\cup\partial B)$.  Since $p\notin \partial B$,  $\Sigma_pB$ is without boundary.   Let $0<\epsilon< \theta_0$.  Since $\dirn_{(p,\epsilon)}W$ is the spherical suspension of $\Sigma_pB$ by Lemma \ref{lem:tan-cone}\,(\ref{f>0-tan}), and hence is without boundary, then $(p,\epsilon)$ is an interior point of $W$. Also $(p,0)\in\partial W$, since  $\dirn_{(p,0)}W$ is the hemispherical suspension of $\Sigma_pB$.
 
There  is $c  > 0$ and a neighborhood  $U$  in $ B-(Z\cup\partial B)$ of $p$ such that if $q\in U$ and $0<\epsilon< c$, the nearest point in $\partial W$ to
 $(q,\epsilon)$  lies in $B\times\{0\}$. It follows, by Theorem \ref{thm:wf-convexity}\,(\ref{d-bdry}) applied to $X=W$, that for any geodesic $\alpha$ in $U$, the restriction of $\sn^\kappa\circ\, \dist_{B\times\{0\}}$ to the geodesic $\alpha\times\{\epsilon\}$ in $B\times\{\epsilon\}$ (necessarily also a geodesic in $W$)  is sinusoidally $\kappa$-concave.  Thus $f\,|\,\bigl(B-(Z\cup\partial B)\bigr)\,\in\,\conc^\kappa$  by Lemma \ref{lem:f-lim}. 

 \begin{claim}\label{cl:Z-in-bdry}
$Z\subset\partial B$.
 \end{claim}
The claim is true if $\dim B=1$.  In that case, $B$ is either a circle, or a closed  interval, i.e. a connected closed subset of $\R$.   If $f ( p ) =0$ for  some $p\in B-\partial B$, then geodesics of $B \times_f J$ bifurcate,
contradicting $B \times_f J\in \CBB^{\,\kappa}$. 
Specifically, we can choose an isometric imbedding  
$\gamma_B:[b, a] \to B$, where $b < 0 < a$,  $\gamma_B(0) = p$, and
$f(\gamma_B(a)) >  0$. By Proposition  \ref{prop:f=0}\,(\ref{f=0-fiber-projection}), we may define $\gamma_F(s)$ when $s \notin Z$ to be $0$ for 
$b \le s \le 0$, and either $0$ or $\theta_0$ for $0 < s \le a$.

Now choose $n>1$, and assume the claim is true whenever $\dim B=n$.
 
Suppose  $\dim B=n+1$,  and $f ( p ) =0$ for  some $p\in B-\partial B$. 
We have
$\Sigma_{\overline p}\,(B\times_fJ)\in \CBB^1$  \cite{bgp}.
By Lemma  \ref{lem:tan-cone}\,(\ref{diff}),  
$$\dirn_{\overline p}\,(B\times_fJ) =\dirn_pB\times_{({\diff}_pf|\dirn_pB)}J,$$  
where $\dim (\dirn_pB)=n$ and $\partial (\dirn_pB)=\emptyset$  since $p\notin\partial B$.  By the induction hypothesis, 
$({\diff}_pf|\dirn_pB)>0$.   Therefore by claim \ref{cl:f>0-concave}, $({\diff}_pf|\dirn_pB)\in\conc^1$.
This is impossible since $({\diff}_pf|\dirn_pB)$ must take a minimum by compactness. 
In this case, some geodesic to the minimum point must extend as a quasigeodesic on which  ${\diff}_pf|\dirn_pB$ becomes negative, a contradiction.  Hence claim \ref{cl:Z-in-bdry}.

\begin{claim}\label{cl:concave}  
 $f\in\conc^\kappa$. \end{claim}
Suppose  $\alpha$ is  a geodesic of $B$. If $\alpha$ has no internal intersection with $\partial B$, then $f\circ\alpha\in\conc^\kappa$  by claims \ref{cl:f>0-concave} and \ref{cl:Z-in-bdry}.
Otherwise,  $\alpha\subset\partial B$.  Let $\hat\alpha$ be a subsegment of $\alpha$ obtained by arbitrarily small shortening at 
either endpoint.  Since $B\in\CBB^\kappa$,
 $\hat\alpha$   is the unique geodesic between its endpoints. Any sequence of geodesics with endpoints  in $B-\partial B$, and 
approaching  the endpoints of $\hat\alpha$, must lie in $B-\partial B$ and converge to
$\hat\alpha$. Therefore $f\circ\hat\alpha$ is sinusoidally $\kappa$-concave and hence so is  $f\circ\alpha$, as claimed.\end{proof}

\begin{defn}\label{def:glue} 
 Suppose $B \times_f F\in \CBB^\kappa$, where $(B,f,F)$ is a WP-triple. Set $Z(f)= f^{-1}(0)$, where  $Z(f)\subset\partial B$ by Theorem \ref{thm:mainB:base-f}. 
Define $B^\dag(f)$ to be  the result of gluing two copies of
$B$ along   $\,\cl(\partial B-Z(f))$.  Define 
$f^\dag: B^\dag(f)\to \R_{\ge 0}$ by $f^\dag=f\circ (\Pi^\dag(f))$ where  $\Pi^\dag(f):B^\dag(f)\to B$ is the tautological map. 
\end{defn}

Now we use Petrunin's incomplete-globalization theorem, Theorem  \ref{thm:anton}, to prove  the following partial-boundary gluing  theorem.  
Since the gluing theorem may be accessed at the level of direction spaces by induction on dimension, the task is to show that it transmits to the underlying space.

\begin{thm}[Theorem  \ref{thm:mainB}\,(\ref{mainB:dag})]
\label{thm:mainB-base-wf}
Let $(B, f,F)$ be a WP-triple. Suppose $B \times_f F\in \CBB^\kappa$. Then
$B^\dag(f)\in\CBB^\kappa$ and  $f^\dag\in \conc^\kappa$.
\end{thm}

\begin{proof}

By Lemma \ref{lem:F=interval}\,(\ref{F=int}), we may assume $F$ is a non-trivial interval $J=[-\theta_0,\theta_0]$.

\setcounter{claim}{0}

Let us write   $Z=Z(f)=f^{-1}(0)$, $B^\dag=B^\dag(f)$ and $\Pi^\dag=\Pi^\dag(f)$. By Theorem \ref{thm:mainB:base-f}\,(\ref{Z-in-bdry}), $Z\subset\partial B$. 

For $a\ge 0$, set 
$$B^\dag_{\,a}=(f^\dag)^{-1}\bigl((a,\infty)\bigr),\,\,\,\,\,\,f^\dag_{\,a}=f^\dag | \,\,B^\dag_{\,a}.$$ 
Then $B^\dag_{\,0}= (\Pi^\dag)^{-1}(B-Z)$.
Let $\overline{{B^\dag}_{\! a}}$ be the closure of   $B^\dag _{\,a}$ in ${B^\dag}$. 

\begin{claim}\label{local-dag}
$\curv B^\dag_{\,0}\,\ge\K$. \end{claim}
Let $B^\ddag$ denote the double of $B$. By Perelman's doubling theorem,  $B^\ddag \in \CBB^{\,\kappa}$   \cite[Theorem 5.2]{Pm1}.  The natural embedding of the space $B^\dag_{\,0}$
in $B^\ddag$ is a local isometry. Hence the claim.

\begin{claim}\label{local-dag1}
$f^\dag_{\,\,0}\in \conc^\kappa$.
\end{claim}

Let $W$  be the preimage  of   $B\times_f[0,\theta_{\,0}]$ under the tautological map 
$$(B \times_f J)^\ddag\to B \times_f J.$$ 
By reflection, 
(as in the proof of  Theorem \ref{thm:mainB:base-f}, Claim \ref{cl:f>0-concave}) 
intrinsic distance in $W$   equals distance in $(B \times_f J)^\ddag$. Since  $(B \times_f J)^\ddag\in\CBB^\K$, then $W\in \CBB^\kappa$. 

For any $q\in B^\dag_{\,0}$, there is $c > 0$ and a neighborhood  $U$ of $q$ in $B^\dag_{\,0}$ such that if $p\in U$ and $0<\epsilon< c$, the nearest point in $\partial W$ to
 $(p,\epsilon)\in W$  lies in $B^\dag_{\,0}\times\{0\}$. It follows, by Theorem \ref{thm:wf-convexity}\,(\ref{d-bdry}) applied to $X=W$, that for any geodesic $\alpha$ in $U$, the restriction of $\sn^\kappa\circ\, \dist_{B^\dag\times\{0\}}$ to the geodesic $\alpha\times\{\epsilon\}$ in $B^\dag_{\,0}\times\{\epsilon\}$  is sinusoidally $\kappa$-concave.  By Lemma \ref{lem:f-lim},
 the claim follows.

\begin{claim}\label{geodesic-space}
The theorem holds if $\dim B=1$.
 \end{claim}

Since $B\in\CBB^\kappa$,  $B$ is a circle of length $\le 2\cdot  \varpi^{\kappa}$ or a closed interval  of length $\le \varpi^{\kappa}$.
By Theorem \ref{thm:mainB:base-f}\,(\ref{Z-in-bdry}), $Z\subset\partial B$.  If $B=\text{circle}$ or $Z=\partial B$, then $B^\dag=B$ and the claim already holds by Theorem \ref{thm:mainB:base-f}.  So we may assume either $B$ is a ray and $Z=\emptyset$, or $B$ is a finite closed interval and  $Z=\emptyset\text{ or an endpoint}$. 

We have $f^\dag_{\,\,0}\in \conc^\kappa$ by  claim  \ref{local-dag1}. If $B$ is a ray, then $\K\le 0$, 
 $B^\dag\in\CBB^\K$, and $f^\dag=f^\dag_{\,\,0}\in \conc^\kappa$. 
 If $B^\dag$ is a circle, then $B^\dag\in\CBB^\K$ and $f^\dag=f^\dag_{\,\,0}\in \conc^\kappa$. If $B^\dag$ is an interval, then $f^\dag\in \conc^\kappa$.  It follows   that $B^\dag$ is an interval of length $\le\varpi^{\kappa}$ and $B^\dag\in\CBB^\kappa$.

\begin{claim}\label{bdryZ}
Choose $n\ge 1$, and assume the theorem holds 
if $\dim B=n$.
Suppose  $\dim B=n+1$.
Suppose $\alpha^\dag:I\to B^\dag$ is a unit-speed geodesic of $B^\dag$ such that $I$ is an interval with $0$ in its interior, and  $\alpha^\dag(0)=p^\dag$ where 
$(\Pi^\dag)(p^\dag)=p$ for some $p\in Z$. 
Then $\alpha^\dag$ lies in
 $(f^\dag)^{-1}(0)=(\Pi^\dag)^{-1}Z$. \end{claim}

 It suffices to prove the claim for  $I=(-\epsilon,\epsilon)$, for some  $\epsilon>0$.
 
 Denote the
 gluing set by  $G=\cl(\partial B-Z)$.
Then $\partial B$ is the disjoint union
\[
\partial B=\inte Z \,\sqcup \, G,
\]
where 
$\inte$ denotes  interior relative to $\partial B$.   The claim is clear if $p\in\inte Z$, so we assume $p\in G$.
 
 We have
 $\Sigma_{\overline p}\,(B\times_fJ)\in \CBB^{\,1}$  \cite{bgp}.
By Lemma \ref{lem:tan-cone}\,(\ref{diff}),$$\dirn_{\overline p}\,(B\times_fJ) =(\dirn_pB)\times_{({\diff}_pf|\dirn_pB)}J,$$
where $\dim (\dirn_pB)=n$.  
Set $ (\Sigma_pB)^\dag = (\Sigma_pB)^\dag({\diff}_pf|\dirn_pB)$ and  
$$\Pi^\dag_p=\Pi^\dag({\diff}_pf|\dirn_pB):(\Sigma_pB)^\dag\to\Sigma_pB.$$
 
By the induction hypothesis, 
\begin{equation}\label{eq:induction}
(\Sigma_pB)^\dag
\in\CBB^{\,1}\text{ and $({\diff}_pf|\dirn_pB)^\dag
\in\conc ^{\,1}$.}
\end{equation}
Then $({\diff}_pf|\dirn_pB)^\dag$ is nonnegative and not identically $0$, hence must take its minimum at a boundary point of $(\Sigma_pB)^\dag$. (Otherwise there would be a quasigeodesic extension, along which  $({\diff}_pf|\dirn_pB)^\dag$ becomes negative,  of a geodesic to a minimum point.)

Moreover, 
 $({\diff}_pf|\dirn_pB)^\dag$ has a unique maximum point $u^\dag_{\text{max}}$ and 
 \begin{equation}\label{eq:<pi/2}
 |u^\dag_{\text{max}} \,u^\dag\,|_{(\Sigma_pB)^\dag} \le\pi/2
 \end{equation}
 for any $u^\dag\in(\Sigma_pB)^\dag$.
By uniqueness,  if $u_{\text{max}}$ is the direction at which  $({\diff}_pf|\dirn_pB)\in\conc ^1$ takes its maximum, 
then $u_{\text{max}}\in\partial (\Sigma_pB)$ and \begin{equation}\label{eq:dag-grad}
u^\dag_{\text{max}}=(\Pi^\dag_p)^{-1}(u_{\text{max}}).
\end{equation}

Let us  write
$$B^\dag=(B\times\{1,2\})/\sim$$
where $(q,1)\sim(q,2)$ if $q\in G$. 
Set  $\alpha_1^\dag=(\alpha^\dag \circ (-\id))|[0,\epsilon)$ and $\alpha_2^\dag=\alpha^\dag|[0,\epsilon)$.

Suppose both $\alpha_i^\dag$  intersect $(\Pi^\dag)^{-1}(G)$ only at $t=0$.  If both lie in $B\times\{1\}$, say, then $\alpha^\dag$ lies in $\partial (B^\dag)$ as desired.  Thus we may suppose the $\alpha_i^\dag$  lie in different copies of $B$. 

Suppose $\alpha_i^\dag$  intersects $(\Pi^\dag)^{-1}(G)$ at some $t\ne 0$, for one or both $i$.  Then we may shorten 
$\alpha_i^\dag$ so that  its endpoints lie on $(\Pi^\dag)^{-1}(G)$.  By
reflecting maximal open segments in one copy of $B$ into the other copy,
 we may  obtain a 
curve with the same endpoints and length as $\alpha^\dag$ and passing through $p^\dag$, and which lies for $t\le 0$ and $t\ge 0$ respectively  in different copies of $B$. 

Therefore we  
 may assume that 
$\alpha_i^\dag$ lies in  $B\times\{i\}$.  

Let $u_i^\dag=(\alpha_i^\dag)^+(0)$, $\alpha_i=\Pi^\dag\circ\alpha_i^\dag$, and  $u_i=\alpha_i^+(0)$.

We may choose a geodesic direction $v\in\Sigma_pB$ arbitrarily close to $u_{\text{max}}$. 
For a geodesic  $\sigma$  with $v=\sigma^+(0)$, let $\sigma_i^\dag$ be the geodesic in $B\times\{i\}\subset B^\dag$ such that $\Pi^\dag\circ\sigma_i^\dag=\sigma$.  Set $v_i^\dag=(\sigma_i^\dag)^+(0)$. 
Then
\begin{equation}\label{eq:dag-angle}
|u_i\,v\,|_{\Sigma_pB}=|u_i^\dag\,v_i^\dag\,|_{(\Sigma_pB)^\dag},
\end{equation}
since the righthand side is at most the lefthand side, and is not  
smaller by the reflection argument.

Now we are going to show 
\begin{equation}\label{eq:sum>pi}
|u_1^\dag\,u^\dag_{\text{max}}\,|_{(\Sigma_pB)^\dag} + |u^\dag_{\text{max}} \,u_2^\dag\,|_{(\Sigma_pB)^\dag}=\pi.
\end{equation}

Let us check that for any $\epsilon >0$, if
 $v$ is sufficiently close to $u_{\text{max}}$ then 
 \[
|u_1\,v\,|_{\Sigma_pB} + |v\,u_2\,|_{\Sigma_pB} \,\ge\, \pi -\epsilon.
\]
Indeed, suppose not.  By Lemma \ref{lem:blowup}, there 
exists $0<b<1$ 
such that 
for some  $v$ arbitrarily close to $u_{\text{max}}$, 
\begin{equation}\label{eq:sum<pi}
|\alpha_1(t)\,\sigma_1(b\cdot t)| + |\sigma_2(b\cdot t)\, \alpha_2(t)|  \,\le 
\,2\cdot
a\cdot t+o(t),
\end{equation}
where $a=\sin\frac{\pi -\epsilon}{2}<1$.
Moreover, since $({\diff}_pf)(u_{\text{max}})>0$, for  $c=\frac{1-a}{3}$ and $v$ sufficiently close to $u_{\text{max}}$,
\begin{equation}\label{eq:close-to-glue}\dist_{(\partial B - Z)} \sigma(b\cdot t) \le c
\cdot  t +o(t).
\end{equation}
 It follows from (\ref{eq:sum<pi})and (\ref{eq:close-to-glue}) that 
 $$|\alpha_1^\dag (t)\,\alpha_2^\dag(t)|\le 2\cdot (a +c) \cdot t  + o(t),$$
where $a+c<1$.  
Then
 a segment of the geodesic $\alpha^\dag$ including $p^\dag$ does not minimize, a contradiction.

Therefore (\ref{eq:sum>pi}) follows from (\ref{eq:dag-grad}) and  (\ref{eq:dag-angle}).
By (\ref{eq:<pi/2}), each term on the lefthand side of  (\ref{eq:sum>pi})  equals $\pi/2$. Then $({\diff}_pf|\dirn_pB)^\dag(u_i^\dag)=0$ by (\ref{eq:induction}). Thus $f\circ\alpha_i=0$ by concavity of $f$, and $\alpha^\dag$ lies in 
 $(f^\dag)^{-1}(0)=(\Pi^\dag)^{-1}Z$ by Theorem \ref{thm:mainB:base-f}\,(\ref{Z-in-bdry}).  

\begin{claim}
The theorem holds in all dimensions.\end{claim}
Choose $n\ge 1$, and assume the theorem holds 
if $\dim B=n$.
Suppose  $\dim B=n+1$.
By claim \ref{bdryZ},  $B^\dag_0$ is a geodesic space.  By Theorem \ref{thm:anton} and claim \ref{local-dag},  $B^\dag\in\CBB^\kappa$. By claims  \ref{local-dag1} and  \ref{bdryZ},  $f^\dag\in\conc^\kappa$.\end{proof}

\begin{rem}
\label{rem:ab}

In the proof of \cite[Theorem 6.2.2,   case $\kappa\le 0$]{ab-wp}, the following argument is outlined.

Given: a continuous function $\varPhi (p,\theta)= f(p)\cos\theta:B\times_f J\to\R_{\ge 0}$  where  $J=[-\pi/2,\pi/2]$,  such that  $\varPhi\in\conc^\K$,  $\curv \Phi^{-1}(0,\infty)\ge\K$, and $\K\le 0$.  

 Prove: $B\times_fJ\in\CBB^\kappa$. 

The argument suggested requires considerable  preparation to fill in. There is a shorter proof by perturbation, similar to the perturbation argument  in \cite{akp1}.

But Theorem \ref{thm:anton} makes argument unnecessary. Since  $\varPhi\in\conc^\K$, then $\Phi^{-1}\bigl((0,\infty)\bigr)$ is convex in $B\times_f J$.
Since   $ B\times_fJ$ is the completion of  $\Phi^{-1}\bigl((0,\infty)\bigr)$,   the claim follows immediately from Petrunin's incomplete-globalization theorem (Theorem \ref{thm:anton}).\end{rem}

\section{Curvature of the fiber, $\CBB$
} 
 \label{sec:fiber-cbb}
 
 This section  finishes the proof of   Theorem  \ref{thm:mainB}, completing our consideration of curvature bounded   below.

 \begin{thm}[Theorem \ref{thm:mainB}\,(\ref{mainB:fiber-f>0})\,\&\,(\ref{mainB:fiber-f=0})] \label{thm:mainB-fiber} 
 Suppose $B \times_fF\in\CBB^\kappa$, where $(B, f,F)$ is  a WP-triple. Set $Z=f^{-1}(0)$.  
   \begin{enumerate}
  \item[(i)]
  If $Z = \emptyset$, then $\kappa\le 0$, and $F\in\CBB^{\,\kappa_F}$ for $\kappa_F=\kappa\cdot (\inf   f)^2$.

\item[(ii)]
If $Z \ne \emptyset$, then 
$F\in\CBB^{\,\kappa_F}$ for

\vspace{1mm}

\hspace*{-15mm}$\kappa_F \ \ =\ \sup\,\{|{\grad_q}f\,|^2: \,q\in Z\}$

\hspace{-10mm}$= \sup\,\{(f\circ\alpha)^+(0)^2 : \alpha =\!\dist_Z\textrm{-realizer with footpoint } \alpha(0)\in Z,\, |\alpha^+(0)|=1\}.$

\end{enumerate}
\end{thm}

\begin{lem}\label{lem:Z-not-empty}
If $B \times_f F\in \CBB^\kappa$ for $\kappa \ge 0$, then one of these statements holds:
  \begin{enumerate}
   \item[(a)]
       $\emptyset\ne Z\subset \partial B$,
 \item[(b)]
$f\equiv a>0$ , $Z=\emptyset$, $ \kappa=0$.
    \end{enumerate}
\end{lem}

\begin{proof} 
By Theorem \ref{thm:mainB:base-f}\,(\ref{Z-in-bdry}), $Z\subset\partial B$.

Let us write   $B^\dag=B^\dag(f)$. 
By Theorem \ref{thm:mainB-base-wf}, we have 
$  B^\dag\in\CBB^\kappa$ and  $f^\dag\in \conc^\kappa$.  Therefore along a quasigeodesic $\alpha^\dag$ in the interior  of $B^\dag$, $f^\dag\circ\alpha^\dag$ is sinusoidally $\kappa$-concave, i.e. 
  its value 
  when $\kappa>0$ (respectively, $\kappa=0$)  is supported from above by a multiple of a translate of 
$\sn^\kappa$ (respectively, by a linear function) 
having  the same initial value and derivative.  (The definition and properties of quasigeodesics are developed in \cite{PP} and  \cite[Chapter 5]{petrunin_survey}.  Also see \cite[p. 1153]{ab-wp} for a discussion of  the support property used here.)

If $f\ne\text{constant}$, this derivative can be taken to be negative.  Then $\alpha$ cannot be continued indefinitely in the interior of $B^\dag$ since $f^\dag\circ\alpha^\dag$ cannot become negative.  Therefore $\alpha$ reaches $\partial B^\dag$, where $f^\dag =0$. Hence $Z\ne \emptyset$.  

If $f\equiv a>0$, then since $f\in \conc^\kappa$ for  $\kappa \ge 0$, we must have $\kappa =0$.
\end{proof}

\begin{proof}[Proof of Theorem \ref{thm:mainB-fiber}]
The theorem is broken into three cases, which are proved in Propositions \ref{prop:B-fiber-inf>0}, \ref{prop:B-fiber-inf=0}, \ref{prop:B-fiber-0} below.
\end{proof}

\begin{prop}\label{prop:B-fiber-inf>0}
Suppose $B \times_f F\in \CBB^\kappa$, where  $\inf f>0$. 
Then $F\in\CBB^{\,\kappa_F}$ for $\kappa_F=\kappa\cdot (\inf  f)^2$.
\end{prop}

\begin{proof}
By the length formula  (\ref{eq:length}), $F$ is closed in $ B \times_fF$ and hence is complete. 

Let us rescale $f$ so that $\inf f = 1$, scaling
the metric of $F$ by the reciprocal factor so as to preserve $W$. 
Choose $p_i \in B$ such that $f(p_i) = 1+a_i$ where $a_i \to 0$.  

For $\varphi ,\psi \in F$, 
$$
|\varphi\,\psi|_F \,\le\, \bigl|(p_i,\varphi)\,(p_i,\psi)\bigr|_{B\times_f F} \,\le\, \bigl|(p_i,\varphi)\,(p_i,\psi)\bigr|_{\{p_i\}\times F}\,=\, (1+a_i)|\varphi\,\psi|_F,
$$ 
where the first inequality is by the length formula (\ref{eq:length}).

Therefore 
$$
\lim_{i\to\infty}\  \bigl|(p_i,\varphi)\,(p_i,\psi)\bigr|_{B\times_f F}\,=\, |\varphi\,\psi|_F.
$$
Since quadruples in $B\times_fF$ satisfy  $(1+3)^\K$, so do  quadruples in $F$. \end{proof}

\begin{prop}\label{prop:B-fiber-inf=0}
Suppose $B \times_f F\in \CBB^\kappa$, where 
$f>0$ and $\inf f=0$.    Then $F\in\CBB^0$.
\end{prop}

\begin{proof}
\setcounter{claim}{0}
By Lemma \ref{lem:Z-not-empty}, $\kappa<0$.

Consider  $p_i \in B$ such that $f(p_i) = a_i \to 0$. Set $$\lambda_i=1/a_i,\quad B_i=\lambda_i\cdot B,\quad f_i=\lambda_i\cdot f:B_i\to\R_{\ge 0}.$$
Then $B_i\times_{f_i}F\in\CBB^{\K_i}$ where $\kappa_i=a_i^2\cdot \kappa$. 
By Theorem \ref{thm:mainB:base-f}, $B_i\in\CBB^{\kappa_i}$, and
 $f_i\in\conc^{\K_i}$. 

Passing to a subsequence, we may assume that the pointed 
spaces $(B_i,p_i)$ have Gromov-Hausdorff limit
$$\lim_{i\to\infty}(B_i,p_i) =(B_\infty,p_\infty).$$

\begin{claim}\label{cl:blowup-inf=0}
For every $\epsilon > 0$ and $r>0$, if $i$ is sufficiently large then 
$$1-\epsilon <f_i\,|\,\overline{\mathcal{B}}(p_i,r) < 1+ \epsilon.$$ 
\end{claim}
Suppose $\alpha$  is a  unit-speed geodesic of length $\le r$ in $B_i$, with $\alpha(0)=p_i$. 
Extend $B_i$ to its double $B_i^\ddag$, where $B_i^\ddag=B_i^\dag(f_i)$ since $Z=\emptyset$.  (If $\partial B=\emptyset$, then $B_i^\ddag=B_i$.) We extend $f_i$ to   $f_i^\dag$, which in this case we write $f_i^\ddag$.
Then $B_i^\ddag\in\CBB^{\kappa_i}$, and  $f_i^\ddag\in\conc^{\K_i}$, by Perelman's doubling theorem and Theorem \ref{thm:mainB-base-wf}. In $B_i^\ddag$, we may extend $\alpha$ to $[0,\infty)$  as a quasigeodesic, on which $f_i^\ddag\circ\alpha$ satisfies the $\kappa_i$-concavity inequality.  
For any $q=\alpha(s_1)$, then $f_i^\ddag\circ\alpha$ 
is supported above by the $\kappa_i$-sinusoid that shares the same  value and derivative  at $s_1$, i.e.
\begin{equation}\label{eq:ub-sinusoid}
0\ <\ (f_i^\ddag\circ\alpha)(s)\ 
\le\  f_i^\ddag(q)\cdot \cosh\bigl(\sqrt{-\kappa_i}\cdot (s-s_1)\bigr)+b\cdot \sinh\bigl(\sqrt{-\kappa_i}\cdot (s-s_1)\bigr),
\end{equation}
where\begin{equation}\label{eq:same-deriv}
b\cdot \sqrt{-\kappa_i}=(f_i^\ddag\circ\alpha)^+(s_1).
\end{equation}

For $b\le 0$, the exponential function
$$f_i^\ddag(q)e^{-\sqrt{-\kappa_i}\cdot (s-s_1)}
$$
 is the extreme  possibility for such a supporting sinusoid that does not vanish on $[s_1,\infty)$.
Therefore for any choice of $\alpha$ and any $s_1\ge 0$ 
\begin{equation}\label{eq:diff-ineq}
(f_i^\ddag\circ\alpha)^+(s_1) \ge -\sqrt{-\kappa_i}\cdot (f_i^\ddag\circ\alpha)(s_1)
\end{equation} 
(regardless of the sign of $b$).
Integrating this differential inequality gives 
\begin{equation}\label{eq:inf-0-lb}
(f_i^\ddag\circ\alpha)(s)\ge e^{-\sqrt{-\kappa_i}\cdot s}.
\end{equation}

Now extend $\alpha$ to $(-\infty,\infty)$ as a quasigeodesic. Since $f_i^\ddag$ is  sinusoidally $\kappa_i$-concave along all of this extension, 
(\ref{eq:diff-ineq})  also holds at $s_1 = 0$  for the left derivative: $$(f_i^\ddag\circ\alpha)^-(0) \ge -\sqrt{-\kappa_i}.$$ The concavity property of
$f_i^\ddag\circ\alpha$ tells us that the sum of its one-sided derivatives at any point is
non-positive; hence 
$(f_i^\ddag\circ\alpha)^+(0) \le -(f_i^\ddag\circ\alpha)^-(0) \le \sqrt{-\kappa_i}$.
Further, by (\ref{eq:ub-sinusoid}) at $s_1 = 0$, 
$$(f_i^\ddag\circ\alpha)(s)\le \cosh(\sqrt{-\kappa_i}\cdot s)+b\cdot \sinh(\sqrt{-\kappa_i}\cdot s)$$
where $b\cdot \sqrt{-\kappa_i}=(f_i^\ddag\circ\alpha)^+(0)$ by (\ref{eq:same-deriv}).
Therefore $b\le 1$, and 
\begin{equation}\label{eq:inf-0-ub}
(f_i^\ddag\circ\alpha)(s)\le e^{\sqrt{-\kappa_i}\cdot s}.
\end{equation}
By (\ref{eq:inf-0-lb}) and (\ref{eq:inf-0-ub}), if $0 < s \le r$, and $\kappa_i$ is
sufficiently close to $0$, then 
$$1-\epsilon \le (f_i^\ddag\circ\alpha)(s)\le 1 + \epsilon.$$
This proves claim \ref{cl:blowup-inf=0}.
 
Fixing $\epsilon$ and $r$ and taking a limit as $i \to \infty$, and then letting $\epsilon \to 0$, gives $f_\infty = 1$ on the ball of radius $r$. Since $r$ is arbitrary, 
$f_\infty = 1$ on  $B_\infty$.

Now choose $\varphi_{\infty}\in F$.  
Set
$$W_i = \lambda_i\cdot (B\times_fF) = B_i\times_{f_i}F.$$
Then
\begin{equation}\label{eq:g-h}
\lim_{i\to\infty}\bigl(W_i,\,(p_i,\varphi_{\infty})\bigr)\ \ =\  \ \bigl(W_{\infty}, \,\,\,(p_\infty,\varphi_{\infty})\bigr)
\end{equation}
where $W_{\infty} = B_{\infty}\times F$.
Indeed, we have $\lim_{i\to\infty}(B_i,p_i) =(B_\infty,p_\infty)$, i.e. given $r>0$ and $\epsilon^\prime >0$, for all $i$ sufficiently large there is $h_i:\mathcal{B}(p_i,r)\to B\,$ satisfying $\dis h_i< \epsilon^\prime$, $f(p_i)=p_{\infty}$, and $\mathcal{B}(p_{\infty},r-\epsilon^\prime)\subset\mathcal{B}(h_i(\mathcal{B}(p_i,r)),\epsilon^\prime )$. By claim \ref{cl:blowup-inf=0}, it follows that $H_i=h_i\times\id:W_i\to W_{\infty}$ satisfies analogous conditions  defining (\ref{eq:g-h}).

Since  $W_i\in\CBB^{\kappa_i}$, where $\kappa_i \to 0$, then $W_\infty\in\CBB^0$. Hence $F\in\CBB^0$.
\end{proof}

\begin{lem}\label{lem:grad-foot}
Suppose $B\in\CBB^\kappa$, and  $f:B \to \R_{\ge 0}$ satisfies $f\in\conc^\kappa$.  
Set $Z=f^{-1}(0) $ and suppose 
$\emptyset\ne Z\subsetneq \partial B$. 
 If $B^\dag(f)\in\CBB^\kappa$ and $f^\dag\in\conc^\kappa$,  then 
\begin{align}
\sup&\,\{|{\grad_q}f\,|^2: \,q\in Z\}\notag\\
\ =&\ \sup\,\{(f\circ\alpha)^+(0)^2 : \alpha =\!\dist_Z \textrm{-realizer with footpoint } \alpha(0)\in Z,\, |\alpha^+(0)|=1\}.\notag
\end{align}
\end{lem}

\begin{proof}
Let us write   $B^\dag=B^\dag(f)$ and $\Pi^\dag=\Pi^\dag(f)$.   

\setcounter{claim}{0} 
\begin{claim}\label{Z=intZ}
 $Z=\cl(\inte Z)$, where 
$\inte$ denotes  interior relative to $\partial B$. 
 \end{claim}

Set $G=\cl(\partial B-Z)$;
then $\partial B$ is the disjoint union
\begin{equation}\label{eq:cl-int-Z}
\partial B=\cl(\inte Z) \,\sqcup \,\inte G.
\end{equation}
This equation is purely topological,
 using only the duality of $\cl$ and $\inte$ via complementation.
The claim is an additional refinement, showing that when 
$B^\dag\in\CBB^\kappa$ and $f^\dag\in\conc^\kappa$, 
then gluing does not hide any vanishing points of $f$,
but rather leaves all points of $Z$ in $\partial (B^\dag)$.

Suppose $p^\dag\in B^\dag$ satisfies $\Pi^\dag(p^\dag)=p\in Z$. 
Since 
${\diff}_{p^\dag}(f^\dag)$ is $\ge 0$ and not identically $0$, there is a geodesic $\alpha^\dag$ in $B^\dag$ with $\alpha^\dag(0)=p^\dag$ and $(f^\dag\circ \alpha^\dag)^+(0)>0$. If $p^\dag\notin\partial (B^\dag)$, there would be a quasigeodesic extension of  $\alpha^\dag$ across $p^\dag$ on which $f^\dag\circ\alpha^\dag <0$  , and this is impossible. 
Therefore \begin{equation}\label{eq:Z-G}
Z\cap\inte G=\emptyset.
\end{equation}
The claim follows from (\ref{eq:cl-int-Z}) and (\ref{eq:Z-G}). 
	
\begin{claim}\label{cl:h}
Suppose $\alpha^\dag:[0,\epsilon)\to B^\dag$ is a unit-speed $\dist_{\partial(B^\dag)}$-minimizer with footpoint $p^\dag=\alpha^\dag(0)\in \partial(B^\dag)$.  Then $\Sigma_{p^\dag}(B^\dag)$
is a hemispherical cone,
$$\Sigma_{p^\dag}(B^\dag) = [0,\pi/2] \times_{\sin\circ\id}\partial(\Sigma_{p^\dag}(B^\dag)),
$$
with vertex  
$u^\dag=(\alpha^\dag)^+(0)$. Moreover,  
$$
u^\dag=
|\grad_{p^\dag}(f^\dag)|^{-1}
\grad_{p^\dag}(f^\dag),
$$
and  $h^\dag = {\diff}_{p^\dag}(f^\dag)\,|\,\Sigma_{p^\dag}(B^\dag)$ satisfies  
\begin{equation}\label{h-grad}
h^\dag=h^\dag(u^\dag)\cdot (\cos\circ\dist_{u^\dag}).
\end{equation}
\end{claim}

The hemispherical cone structure of $\Sigma_{p^\dag}(B^\dag)$ at a footpoint in
$\partial (B^\dag)$ is derived in \cite{Pm1}; Êit is a direct corollary of the doubling
theorem. ÊThe vertex of that cone is, as claimed, $u^\dag=(\alpha^\dag)^+(0)$.

We have Ê$h^\dag\in\conc^1$ and $h^\dag\ge 0$. ÊLet $v^\dag$ be unique maximum
point of $h^\dag$. ÊLet $\sigma^\dag:[0,\ell]\to \Sigma_{p^\dag}(B^\dag)$  be a unit-speed local geodesic satisfying $\sigma^\dag(0)=v^\dag$ and passing through  $u^\dag$,
extending to length $\ell \ge \pi/2$
before  terminating at the equator.
Then $(h^\dag\circ\sigma^\dag)(s)\le h^\dag(v^\dag) \cdot \cos s$.  Therefore $v^\dag=u^\dag$.    (\ref{h-grad}) holds since $s= \dist_{u^\dag}(\sigma^\dag(s))$.
 
\begin{claim}\label{footpt-grad}
Suppose $\alpha:[0,\epsilon)\to B$ is a unit-speed $\dist_Z$-minimizer with footpoint $p=\alpha(0)\in Z$. Set $u=\alpha^+(0)$.  Then $u = |\grad_{p}f|^{-1}
\grad_{p}f$.
\end{claim}

If $p\in\inte Z$, claim \ref{footpt-grad}
follows from claim \ref{cl:h}. 

Suppose $p\notin\inte Z$.  Let $\Pi^\dag\circ\alpha^\dag=\alpha$ and set $\alpha^\dag(0)=p^\dag$. By claim \ref{Z=intZ}, $ p^\dag\in\partial (B^\dag)$.
 By reflection, $\alpha^\dag$ satisfies claim \ref{cl:h}.  Claim \ref{footpt-grad} follows.
	
	\begin{claim}\label{lim-footpt}
Any $q\in Z$ is the limit of 
$\dist_Z$-footpoints $p\in\inte Z$, i.e. points
$p\in \inte Z$ such that $p$ is the footpoint of a $\dist_Z$-minimizer.	\end{claim}

		Any $q\in\inte Z$ is the limit of $\dist_Z$-footpoints in $\inte Z$. 
Indeed, we may choose a curve $\alpha:[0,\epsilon)\to B$  with $\alpha(0)=q$ and $\alpha(t)\in B-\partial B$ for $t>0$. For $t$ sufficiently close to $0$, $\alpha(t)$ has $\dist_Z$-footpoints $p\in \inte Z$ arbitrarily close to $q$. 

Therefore claim \ref{lim-footpt} follows from  claim \ref{Z=intZ}.	
	
\begin{claim}
The lemma follows from claims \ref{footpt-grad} and \ref{lim-footpt}.
\end{claim}

It is straightforward to show (as in 
\cite[Lemma 1.3.4]{petrunin_survey})
 that the function
 $|{\grad_q}f|$
 is lower semicontinuous on $B$, i.e. for any sequence $q_i\to q\in B$, 
 $$ |{\grad_q}f|\ \le\ \liminf_{i\to\infty} |{\grad_{q_i}}f|.$$
\end{proof}

\begin{prop}\label{prop:B-fiber-0}
Suppose $B \times_f F\in \CBB^\kappa$, where   
$Z\ne\emptyset$.  Then $F\in\CBB^{\,\kappa_F}$, where
\begin{align}
\kappa_F \ \ =\ \sup\,\{|{\grad_q}f\,|^2: \,q\in Z\}\hskip67mm
\label{eq:sup-grad}\\=\ \sup\,\{(f\circ\alpha)^+(0)^2 : \alpha =\!\dist_Z \textrm{-realizer with footpoint } \alpha(0)\in Z,\, |\alpha^+(0)|=1\}.\notag
\end{align}
\end{prop}

\begin{proof}
By Lemma \ref{lem:grad-foot}, it suffices to verify the first equality in   (\ref{eq:sup-grad}). 

\setcounter{claim}{0} 
\begin{claim}
Proposition \ref{prop:B-fiber-0} holds for warped products with $1$-dimensional base.
\end{claim}
In this case,  $B$ is isometric to a closed interval.   If $p \in Z$, then $p$ is an endpoint of $B$ by Theorem \ref{thm:mainB:base-f}\,(\ref{Z-in-bdry}), so $\Sigma_pB=\{u\}$.  By Lemma  \ref{lem:tan-cone}\,(\ref{diff}),
 $$\Sigma_{\overline p}(B \times_f F) \,= \,\{u\} \times_a \,F \,\cong \,a\cdot F,$$ where $a={\diff}_pf\,(u)=|{\grad_p}f\,|$.  Therefore $a\cdot F\in\CBB^1$, so
$F\in\CBB^{a^2}$.

\begin{claim}[Induction step]
Suppose 
(\ref{eq:sup-grad}) holds for warped products with $n$-dimensional base.  Then  (\ref{eq:sup-grad}) 
holds for warped products with $(n+1)$-dimensional base.
\end{claim}

Let  $\dim B=n+1$.

Any $q\in Z$ is the limit of $\dist_Z$-footpoints $p\in\inte Z$, by claim \ref{lim-footpt} of the proof of Lemma \ref{lem:grad-foot}. By 
lower semicontinuity of $|{\grad_q}f|$, it suffices to restrict the supremum in the first equality in (\ref{eq:sup-grad}) to $\dist_Z$-footpoints $p\in \inte Z$.

Set $ h = {\diff}_pf\,|\,\Sigma_pB$. 
Then $\Sigma_pB \times_hF\in\CBB^1$ by Lemma  \ref{lem:tan-cone}\,(\ref{diff}).  Since  $\dim\Sigma_pB =n$, the induction hypothesis implies   $F\in\CBB^{\,\kappa_F}$ where 
 $$\kappa_F =\sup\{{|\grad_v}h\,|^2: \,v\in \partial(\Sigma_pB)\}.$$
Since 
$$|\grad_vh|=|\grad_pf|$$
 for any $v\in \partial(\Sigma_pB)$ by (\ref{h-grad}),  
this completes the induction step.\end{proof}

\section{Curvature of the fiber, $\CAT$
} 
 \label{sec:fiber-cba}

This section  finishes the proof of   Theorem  \ref{thm:mainA}, completing our consideration of curvature bounded  above.

In a Riemannian warped product $B\times_fF$,  the vertical leaves  $\{p\} \times F$    
are umbilic, 
with extrinsic curvatures
$$|\grad_pf|\,/f(p),$$ 
i.e.
for a geodesic $\beta$ in $F$, the curve $(p, \beta)$ has curvature $|\grad_pf|\,/f(p)$ at every point.

Since the
acceleration of an intrinsic geodesic in a vertical
leaf is towards the lower values of the warping function, the intuition behind this formula is that we actually need the downward gradient length $|\grad_p(-f)|$, which however agrees with $|\grad_pf|$ in Riemannian manifolds. This agreement need not occur in $\CAT$ spaces, so we expect the downward gradient to appear, as in Lemma \ref{lem:fiberbound}.

In metric spaces, a theory of curvature of curves was developed in \cite{ab-cbc}. 
Building on work of Lytchak \cite{lytchak-reach}, a ``Gauss equation'' for $\CAT^\K$ spaces was proved in \cite{AB-gauss}, i.e. a sharp upper curvature
bound on a subspace whose intrinsic geodesics have an
extrinsic curvature bound. Now we are going to apply this work to  obtain the correct curvature bound for the  fiber in a $\CAT^\K$ warped product.

\begin{defn}[Extrinsinc curvature]\label{def:ext-curv} Suppose $Y\subset X$, where $X$ is an intrinsic metric space and the intrinsic metric induced on $Y$ is complete.  
Then $Y$ is a subspace of 
\emph{extrinsic curvature $\le A$},
where $A\ge 0$,  
if intrinsic distances $\rho$ in $Y$ 
and extrinsic distances $s$ in $X$ satisfy
\begin{equation}\label{eq:ext}
\rho-s \le (A^2/24)\cdot  s^3+ o\,(s^3) 
\end{equation}
on all pairs of points having $\rho$ sufficiently small. 
\end{defn}

\begin{rem}
A Riemannian submanifold has extrinsic curvature $\le A$ if and only if its second fundamental form $II$  satisfies $|II| \le A$.
\end{rem}

\begin{thm}[Gauss equation \cite{AB-gauss}]
\label{thm:gauss-eq}
Suppose $X\in \CAT^\kappa$.  Let $Y\subset X$ be a subspace of 
extrinsic curvature $\le A$.  Then $\curv Y \le \kappa + A^2$.
 \end{thm}  

 In light of the Gauss equation, we need to establish a sharp bound on extrinsic curvature of a vertical leaf.
 
\begin{lem}\label{lem:fiberbound}
Let $(B, f,F)$ be a WP-triple. Suppose  $B \times_fF\in\CAT^\kappa$, where $\K\le 0$. If $f(p) > 0$ and $| \nabla_p(-f) |\ne 0$, then the  vertical leaf $\{p\} \times F$    in  $B\times_fF$ 
has extrinsic curvature 
$$\le | \nabla_p(-f) |/f(p).$$ 
\end{lem}

\begin{lem}\label{lem:fibergeo}
Let $(B, f,F)$ be a WP-triple. Suppose  $B \times_fF\in\CAT^\kappa$. Let $\gamma=(\gamma_B,\gamma_F)$ be a (necessarily unique) geodesic of $B \times_fF$ with endpoints    $(p,\varphi)$, $(p,\psi) \in \{p\}\times F$, where                        $|((p,\varphi)\,(p,\psi)|_{B \times_fF} <\varpi^\kappa$.
Then $\gamma_F$  is the unique pregeodesic of  $F$ with endpoints  $\varphi$, $\psi$.
\end{lem}

\begin{proof}
By Proposition   \ref{prop:fiber-independence}\,(\ref{geo-fiber-proj}), $\gamma_F$ is a pregeodesic of $F$  joining $\varphi$ and $\psi$.   By Proposition   \ref{prop:fiber-independence}\,(\ref{geo-fiber-indep}), for any pregeodesic $\beta$ of $F$ joining $\varphi$ and $\psi$ there is a geodesic of $B\times_fF$ with endpoints $(p,\varphi)$ and $(p,\psi)$ that projects to a monotonic reparametrization of $\beta$. Thus $\gamma_F$ is unique because $\gamma$ is unique.
\end{proof}

\begin {defn}
For $a>0$, set
\begin{equation}\label{eq:cone-def}
\cone_a=\R_{\ge 0}\times_{a\cdot\id}\R.
\end{equation}
\end{defn}

\vspace{1mm}

\begin{proof}[Proof of Lemma \ref{lem:fiberbound}]

\setcounter{claim}{0}

\begin{claim}\label{R-vertical-curv}
It suffices to take $F=\R$, i.e. to show that 
the vertical leaf \,$\{p\}\times \R$\, in  \,$B\times_f\R$\,
has extrinsic curvature 
$$\le  | \nabla_p(-f) |/f(p).$$ \end{claim} 
By Lemma \ref{lem:fibergeo},    $F$ is a locally geodesic space.
By Proposition \ref{prop:fiber-independence}\,(\ref{dist-fiber-indep}),  
vertical leaves are umbilic, 
i.e. if two points of  $\{p\} \times F$ have the same intrinsic distance in $\{p\} \times F$, then they have the same extrinsic distance in $B\times_fF$. It follows that we need only verify 
the extrinsic curvature definition (\ref{eq:ext}) for endpoint pairs lying on a single geodesic in $\{p\}\times F$.

Let $\beta: J\to F$ be a  unit-speed geodesic. By Lemma \ref{lem:F=interval}\,(\ref{F=int:2-parameter}),  under the embedding
$$\id\times\beta: B\times_fJ\to  B \times_fF,$$  the intrinsic and extrinsic metrics of $B\times_fJ$ agree. 
The claim follows.

\begin{claim}\label{fibersymmetry}
Let 
\,$\gamma=(\gamma_B,\gamma_\R):[-s_0,s_0]\to B\times_f\R $\,   be a
 geodesic  with endpoints $(p,\pm\, \psi_0)$.  Then
\,$\gamma_B(s)=\gamma_B(-s)$, and the speed $v_\R$ of $\gamma_\R$ satisfies  $v_\R(s)=v_\R(-s)$ for $s\in [0,s_0]$. 
(The lemma concerns limits as $\psi_0\to 0$, and in its proof we will be taking $\psi_0$ sufficiently small.)
\end{claim}

The claim follows from Corollary \ref{cor:vertical}\,(\ref{F-midpt})  and  uniqueness of  $\gamma$.

\begin{claim}\label{<conelength}
The geodesic of $B\times_f\R$  joining 
$(p,-\psi_0)$ and $(p,\psi_0)$, $\psi_0>0$, 
does not lie in the vertical leaf $\{p\}\times \R$.  
\end{claim}

We are going to use the data from a cone geodesic $\gamma$ to construct a shorter curve  in $B\times_f \R$, 
specifically a curve whose projection to $B$ runs back and forth along a geodesic pointing in a direction of decreasing $f$.

Suppose $0<a<| \nabla_p(-f) |$, and set $r=f( p )/a$.
Let 
$$\gamma=(\gamma_{[0,r]}\,,\gamma_\R):[-s_0,s_0]\to {\cone}_a $$ 
be a unit-speed geodesic with endpoints 
$(r,\pm\psi_0)$. 
Here we write $\gamma_{[0,r]}$ to emphasize that  the projection of the cone geodesic $\gamma$ to the 
base $\R_{\ge 0}$ of ${\cone}_a $ lies in $[0,r]$.

 Since the sector of $\cone_a$ with angle at the vertex $2\psi_0$ is isometric to a sector of the
Euclidean plane of the same angle, 
if $\psi_0<\pi$ then $\gamma$ may be viewed simply as a Euclidean segment
connecting two points of a central circular arc. Hence its projection to the base behaves as described in claim 2 and has speed $0$ only at $s=0$.

We may choose a unit-speed geodesic $\alpha:[0, t_0)\to B$ with $\alpha(0)=p$ and such that $$(f\circ\alpha)(t)\,<\,f( p ) - a\cdot t.$$

For some $\epsilon >0$, if $0<\psi_0<\epsilon$
then
  $r - \gamma_{[0,r]}(0)< t_0$.  Then we may define a curve $\check\gamma : [-s_0,s_0]\to B\times_f \R$ by 
$$\check\gamma(s)=\bigl( \alpha\bigl(r - \gamma_{[0,r]}(s)\bigr),\ \gamma_\R(s)\bigr).$$
By the length formula (\ref{eq:length}), 
$$\length \check\gamma \,<\, \length \gamma\,<\,	
 2\cdot f( p)\cdot \psi_0.$$
Since \,
$2\cdot f( p)\cdot \psi_0$\, 
 is the length of the geodesic  joining 
 $(p,-\psi_0)$ and $(p,\psi_0)$in the vertical leaf $\{p\}\times \R$, the claim follows.

\begin{claim}\label{convex-on-traj}
Let \,$\gamma=(\gamma_B,\gamma_\R):[-s_0,s_0]\to B\times_f\R$\, be a unit-speed geodesic.
Then:
\begin{enumerate}
\item[(i)]
The arclength parameter  $t\in [-t_0,t_0]$
 of $\gamma_B$ is a strictly increasing function $t=t(s)$ of  
the arc-length parameter $s\in [-s_0,s_0]$  of $\gamma$.
\item[(ii)]
$(f\circ\gamma_B)(s(t))$ is a convex function of the arc length parameter $t$ of $\gamma_B$.
\end{enumerate}
\end{claim}

(i) follows from claim \ref{<conelength}.

If $(f\circ\gamma_B)(s(t))$ is not convex, then its restriction to some subinterval $I$  of $[-t_0,t_0]$ is \,$\ge$\, the linear function of \,$t$\, with the same endpoint values. Moreover, $\gamma_B|\, I$\, is not a geodesic since $f$ is convex.   Let $\alpha$ be the geodesic of $B$ joining the endpoints of  $\gamma_B|\, I$\, and parametrized by $I$.  Since $\alpha$ is shorter than $\gamma_B| I$, and $f\circ\alpha$ is convex, then the length formula shows that $\gamma|\,I$\, can be shortened in \,$B\times_f\R$.  This contradiction proves (ii).

\begin{claim}
Let $\gamma=(\gamma_B,\gamma_\R):[-s_0,s_0]\to B\times_f\R$ be a unit-speed geodesic with endpoints $(p,\pm\, \psi_0)$. Then there exists $0<A\le | \nabla_p(-f) |/f(p)$ such that
\begin{equation}\label{eq:fiberbound}
2\cdot f( p )\cdot\psi_0 - 2\cdot s_0\ \le\  (A^2/24)\cdot(2\cdot s_0)^3 + o(s_0^3).
\end{equation}
\end{claim}

In contrast to the proof of claim \ref{<conelength}, here we  use the data from a geodesic in $B\times_f \R$ to construct a shorter curve in a cone. 

As before, let $t\in [-t_0,t_0]$ be the arc-length parameter of $\gamma_B$. By Corollary \ref{cor:vertical}\,(\ref{min}) and (\ref{speed-0}), $v_B$ vanishes only when  $f\circ\gamma_B$
 takes its minimum value, and in particular at $s=0$. Moreover the minimum value occurs only at $s=0$. 
Otherwise, since $(f\circ\gamma_B)(s(t))$ is convex by claim \ref{convex-on-traj},  then $(f\circ\gamma_B)(s(t))$ would take its minimum on a nontrivial interval $I$.  By claim \ref{fibersymmetry}, $I$ would be symmetric about $0$.  Then  $\gamma_B|\,I$ would be constant, since otherwise $\gamma$ could be shortened by replacing $\gamma_B|\,I$\, with a constant curve. But $\gamma_B|\,I$ cannot be constant by claim \ref{<conelength}.  

For a given  $a >0$, we may reduce \,$\psi_0$\, if necessary so that $t_0<r=f( p )/a$.  Define a curve
$\check\gamma:[-s_0,s_0]\to\cone_a$\, with endpoints $(r,\pm\,\psi_0)$, by requiring the projections $\check\gamma_{[0,r]}$ and $\check\gamma_{\R}$ of $\check\gamma$ on base and fiber to have speeds $v_B$ and $v_\R$ respectively. By the length formula (\ref{eq:length}), if
$$f( p )-a\cdot t \,=\, a\cdot  (r-t)\le (f\circ\gamma_B)(s(t)),\ \ 0\,\le\, t\le t_0,$$
then $\length\check \gamma \le\length \gamma$.  In particular, appealing to claim \ref{convex-on-traj}, let us  take
$$a \,=  \, \bigl |\frac{d\,(f\circ\gamma_B)(s(t))}{dt}((-t_0)^+)\bigr|\,\le\,| \nabla_p(-f) |.$$ 

Now we compare the respective curvatures of vertical leaves in $B\times_f\R$ and $\cone_a$. 
Since \,$2\cdot f( p)\cdot \psi_0$\, is the distance in the vertical leaf of $B\times_f \R$ between  $(p,-\psi_0)$ and $(p,\psi_0)$, the curvature formula (\ref{eq:ext}) gives
\begin{equation}\label{eq:>-cone-curv}
2\cdot s_0 = \length \gamma \ge \length \check\gamma \ge 2\cdot f(p)\cdot \psi_0 -
(A^2/24)\cdot (2\cdot f(p)\cdot \psi_0)^3 + o(\psi_0^3),
\end{equation}
where $A=a/f( p )=1/r$ is the curvature of the fiber
$\{t_0\} \times\R$\, in $\cone_a$. 

In the limit as $\psi_0\to 0$, the inequality for
\,$(2\cdot f(p)\cdot \psi_0 - 2\cdot s_0)$\, given by (\ref{eq:>-cone-curv}) yields  (\ref{eq:fiberbound}), proving the claim.

Lemma \ref{lem:fiberbound} follows from  (\ref{eq:fiberbound}), claim \ref{R-vertical-curv} and  Definition \ref{def:ext-curv}. 
\end{proof}

\begin{lem}\label{lem:cat-grad-foot}
Suppose $B\in\CAT^\kappa$, and  $f:B \to \R_{\ge 0}$ satisfies $f\in\conv^\kappa$.  Set $Z=f^{-1}(0) $ and suppose 
$\emptyset\ne Z\subsetneq B$. Then:  \begin{align}
\label{eq:cat-grad-foot}
\inf\,&\{(f\circ\alpha)^+(0)^2 : \alpha =\!\dist_Z \textrm{-realizer with footpoint } \alpha(0)\in Z,\, |\alpha^+(0)|=1\}\\
&= \liminf_{\epsilon \to 0}\ \{|\nabla_p(-f)|^2 : 
0<\dist_Z(p) \le \epsilon\}.\notag\end{align}
\end{lem}

\begin{proof}
\setcounter{claim}{0}
Let $p$ satisfy $0<\dist_Z( p)\le \epsilon$, and $ \alpha $ be a geodesic realizing  $\dist_Z( p)$ with footpoint $\alpha(0)\in Z$,  $ |\alpha^+(0)|=1$. 

\begin{claim}\label{cl:F-bds1}
$(f\circ\alpha)^+(0)^2\,\ge\, C_\epsilon^2$ where 
\[C_\epsilon\,=\,\inf\,\{|\nabla_q(-f)|: 
0<\dist_Z(q) \le \epsilon\}.
\]
\end{claim} 
The claim is trivial if $C_\epsilon=0$, so assume $C_\epsilon> 0$.  
Let $\eta_t$ be the unit-speed downward gradient curve of $f$ starting at
$\alpha(t)$. Then for \,$t$\, sufficiently small,
\,$\eta_t$\, remains within distance $\epsilon$ of $Z$  since 
\[
(f \circ \eta_t)^+ = - | \nabla_{\eta_t}(-f) | \le -C_\epsilon. \]
Indeed, if $t<\epsilon/2$ and $f(\alpha(t))<C_\epsilon\cdot\epsilon/2$, then $f\circ\eta_t$ reaches $0$ before $\dist_Z\circ\,\eta_t$  can exceed $\epsilon$.

Let $s(t) \ge t$ be the length of $\eta_t$.  
Then
\begin{align}
f(\alpha(t)) = \int_{[0, s(t)]} | \nabla_{\eta_t(u)}(-f) | du \ge s(t)\cdot C_\epsilon\ge\  t\cdot C_\epsilon.
\end{align}
The claim follows.

\begin{claim}\label{cl:F-bds2}
$(f\circ\alpha)^+(0)^2\,\le \,|\nabla_{p}(-f)|^2/\cos^2\epsilon$.
\end{claim}
We may assume for this claim that $\kappa = 1$, since scaling changes both sides of the
inequality by the same positive factor. 

By  $(f\circ\alpha)'$ we mean the left-sided or the right-
sided derivative,
which are equal a.e. by semiconvexity of $f$. For the same reason, $(f\circ\alpha)'$ is continuous except for countably many upward jumps.
When $(f\circ\alpha)' > 0$, the sinusoidal $1$-convexity of $f$ implies
that\, $(f\circ\alpha)']^2 + (f\circ\alpha)^2$\, is nondecreasing. 

Let us abbreviate
\,$f \circ \alpha$\, by $f$. Suppose $p=\alpha(t_0)$;  then $t_0\le\epsilon$.   
For $t_0<\pi$,
the $2$-point sine curve bounding $f$
above on $[0,\overline t\,]$, $0<\overline t<t_0$,  is
$$\frac{f(\overline t)}{\sin \overline t}\cdot\sin t.$$
Then  $$f'(\overline t) \ge \frac{f(\overline t)}{\sin\overline t}\cdot\cos\overline t,
\ \textrm{ that is, }
\ f(\overline t) \le \frac{\sin\overline t}{\cos\overline t}\cdot f'(\overline t).$$
In particular, $f'(\overline t) > 0$ for $t_0<\pi/2$.

The claim follows:
\begin{align*}
f^+(0)^2 &\le \lim_{\overline t\to t_0}\bigl(f'(\overline t)^2 + f(\overline t)^2\bigr)\\
	&\le\lim_{\overline t\to t_0}\bigl( f'(\overline t)^2/\cos^2\overline t\bigr)\\
	&\le|\nabla_{p}(-f)|^2/\cos^2t_0.
\end{align*}

The lemma follows from Claims \ref{cl:F-bds1} and \ref{cl:F-bds2}.
\end{proof}

\begin{thm}[Theorem \ref{thm:mainA}\,(\ref{mainA:fiber-f>0})\,\&\,(\ref{mainA:fiber-f=0})]\label{thm:A-fiber} 
Suppose $B \times_fF\in\CAT^\kappa$, where $(B, f,F)$ is  a WP-triple. Set $Z=f^{-1}(0)$.   
\begin{enumerate}

\item[(i)]
If $Z = \emptyset$, then $F\in\CAT^{\,\kappa_F}$
for $\kappa_F=\kappa\cdot (\inf\,f)^2$.

\item[(ii)]\label{cba:fiber-f=0}

If $Z \ne \emptyset$, then $F\in\CAT^{\,\kappa_F}$
for
$\kappa_F=
\min\,\{\kappa_{\textrm{foot}},\,\kappa_{\textrm{far}}\}$, where

\vspace{-5mm}

\begin{align*}
\kappa_{\textrm{foot}} &=\inf\,\{(f\circ\alpha)^+(0)^2 :  \alpha =\!\dist_Z \textrm{-realizer with footpoint } \alpha(0)\in Z,\, |\alpha^+(0)|=1\}
\notag\\
&= \liminf_{\epsilon \to 0}\ \{|\nabla_p(-f)|^2 : 
0<\dist_Z(p) \le \epsilon\},\notag\\
&\ \ \ \ \ \ \ \ \ \ \ \ \ \ \kappa_{\textrm{far}}= \inf \{\kappa\cdot f(p)^2 : \ \dist_Z( p ) \ge \varpi^\kappa/2\}.
\notag
\end{align*}

\end{enumerate}
\end{thm}

\begin{proof}
\setcounter{claim}{0}
\begin{claim}\label{pt-F-comparison}
Suppose $\kappa \le 0$. If $ f(p)>0$, then
\[
\curv F \le \kappa\cdot  f(p)^2 + 
 | \nabla_p(-f) |^2.
\]
\end{claim}

If \,$ | \nabla_p(-f) |= 0$, then $p$ is a minimum point of $f$ since $f$ is convex.  By Proposition \ref {prop:wp-proj}\,(\ref{f-min}),  the intrinsic and extrinsic metrics of \,$\{p\} \times F$ agree.
Hence $\{p\} \times F\in\CAT^{\,\kappa}$, and  the claim follows by scaling.

Suppose \,$ | \nabla_p(-f) |> 0$.  By Lemma \ref{lem:fiberbound} and Theorem \ref{thm:gauss-eq},   $$\curv\,(\{p\} \times F)\le \bigl(\kappa \cdot f( p )^2 + | \nabla_p(-f) |^2\bigr)/f(p)^2.$$
Hence the claim.

\begin{claim}\label{F-comparison-0}
Suppose $\kappa \le 0$. 
If $Z \ne \emptyset$, then 
$$\curv F \ \le\  \liminf_{\epsilon \to 0}\ \{|\nabla_p(-f)|^2 : 
0<\dist_Z(p) \le \epsilon\}.$$

\end{claim}
The claim follows from  claim \ref{pt-F-comparison}.
\begin{claim}\label{F-comparison-inf} 
Suppose $\kappa \le 0$. 
If  $Z=\emptyset$, then $\curv F \le \kappa\cdot (\inf\,f)^2$.
\end{claim}

Consider the sublevel sets
$$S_i = \{p\in B: f( p )\le( \inf f) + 1/i\}.$$  Choose $C>0$.  If $i$ is sufficiently large, there is some $p_i\in S_i$ such that $|\nabla_{p_i}(-f)|\le C$.  Indeed, suppose not.  If $\eta$ is a downward gradient curve of $f$ starting at a point of $S_i$, then $f\circ\eta$ must take values $\le\inf f$, a contradiction.  Therefore  this claim follows from claim \ref{pt-F-comparison}.

\begin{claim}\label{F-comparison}
Let $\kappa$ be arbitrary. 

\begin{enumerate}

\item[(i)]
If $Z = \emptyset$, then $\curv F\le \kappa_F$
for $\kappa_F=\kappa\cdot (\inf\,f)^2$,

\item[(ii)]

If $Z \ne \emptyset$, then $\curv F\le \kappa_F$
for
$\kappa_F=
\min\,\{\kappa_{\textrm{foot}},\,\kappa_{\textrm{far}}\}$, where

\vspace{-5mm}

\begin{align*}
\kappa_{\textrm{foot}} &=\inf\,\{(f\circ\alpha)^+(0)^2 :  \alpha =\!\dist_Z \textrm{-realizer with footpoint } \alpha(0)\in Z,\, |\alpha^+(0)|=1\}
\notag\\
&\ \ \ \ \ \ \ \ \ \ \ \ \ \ \kappa_{\textrm{far}}= \inf \{\kappa\cdot f(p)^2 : \ \dist_Z( p ) \ge \varpi^\kappa/2\}.
\notag
\end{align*}
\end{enumerate}
\end{claim}

By claims \ref{F-comparison-0} and \ref{F-comparison-inf} and Lemma \ref{lem:cat-grad-foot}, the claim is true if $\K\le 0$. So assume $\kappa > 0$, and without loss of generality, set $\kappa = 1$.

Let $\cone f$ be the homogeneous linear extension of $f$ to $\cone B$. Then
\begin{align*}
\cone (B\times_fF)\,&=\, 
\R_{\ge
0}\times_{\id}(B\times_fF)\\
&=\, (\R_{\ge 0}\times_{\id}B)\times_{\cone f}F \, =\, \cone B\times_{\cone f}F.
 \end{align*}
 
We have $\cone(B \times_fF)\in \CAT^0$ (see Remark \ref{rem:cone}\,(b)). Moreover, since  $f\in \conv^1$, then $\cone f\in \conv^0$ (see \cite[Lemma 3.5]{ab-cbc}).
We also have
$$Z(\cone f) \,=\, \cone Z(f)\cup o \,\ne\, \emptyset.$$

Suppose  $\check\alpha:[0,a]\to\cone B$ is a geodesic realizing 
$\dist_{Z(\cone f)}(p,1)$ for some $p\in B$,  with footpoint $\check\alpha(0)\in Z(\cone f)$ and $|\check\alpha^+(0)|=1$.  By claim \ref{F-comparison-0} and Lemma \ref{lem:cat-grad-foot}, 
\begin{equation}\label{eq:cone-geo}
\curv F \le\bigl((\cone f)\circ\check\alpha\bigr)^+(0)^2.
\end{equation}

Suppose  $Z=Z(f)=\emptyset$. Then $Z(\cone f)=\{o\}$, $\check\alpha(0)=o$, and $\check\alpha$ is the radial segment from $o$ to $(p,1)$.  Thus $\bigl((\cone f)\circ\check\alpha\bigr)^+(0) = f( p )$.  (i) follows
 from (\ref{eq:cone-geo}) by rescaling.
 
On the other hand, suppose  $Z=Z(f)\ne \emptyset$. 

If $\dist_Z( p )\ge\pi/2$, we again have  $\check\alpha(0)=o$, so  $\check\alpha$ is the radial segment from $o$ to $(p,1)$. In this case,  $\bigl((\cone f)\circ\check\alpha\bigr)^+(0)=f(p)$ and \,$\curv F\le\kappa\cdot f(p)^2$\, by (\ref{eq:cone-geo}) and rescaling.  Therefore $\curv F\le \kappa_{\textrm{\,far}}$.
 
 Suppose $\dist_Z( p )<\pi/2$. Let $\alpha$ be a geodesic in $B$ realizing \,$\dist_Z( p)$, with footpoint\, $\alpha(0)\in Z$, $|\alpha^+(0)|=1$. Let
  $\cone\alpha$ be the cone  over the image of $\alpha$.  If $\alpha$ has arclength parameter $\theta$, then  the intrinsic and extrinsic metrics of $\cone\alpha$ agree and are isometric to  a sector  of $\E^2$ with polar coordinates
$(r,\theta)$.  In these coordinates,  $(\cone f)\,|\,\cone\alpha\,=\,r\cdot f(\theta)$.  We choose  $\check\alpha$ to lie in $\cone\alpha$, projecting to a reparametrization of $\alpha$.  (It is true that $\alpha$ and $\check\alpha$ are uniquely determined by $p$, but we do not use this fact.) A simple calculation in polar coordinates gives
\[
\bigl((\cone f)\circ\check\alpha\bigr)^+(0)=(f\circ\alpha)^+(0).
\]
Thus (ii) follows from (\ref{eq:cone-geo}) and rescaling. 

\begin{claim}\label{fiber-geos}
For $\epsilon$ sufficiently small, any $\varphi,\psi\in F$ such that $|\varphi\,\psi|_F < \varpi^{\kappa_F + \epsilon}$ 
are joined by a unique geodesic of $F$, and these geodesics depend continuously on $\varphi,\psi$.
\end{claim}

Let $\K_F$ be as in claim \ref{F-comparison}
. It follows from  claim \ref{F-comparison} that either
\begin{enumerate}
\item[(a)]
$\kappa\le 0$, or 
\item[(b)]
$\K>0$ and there is $p\in B-Z$ such that $\kappa_F + \epsilon\ge \kappa\cdot  f(p)^2$. 
\end{enumerate}
Indeed, in case (b), if  $Z\ne\emptyset$ then $\K_F\ge 0$ and \,$\K\cdot f( p)$\, may be taken to be positive and arbitrarily close to $0$.

In case (a), $B \times_fF \in\CAT^{\,0}$, and so $(p,\varphi), (p,\psi)\in \{p\}\times F$ are joined by a geodesic $\gamma$  in $B \times_fF$ that  depends uniquely and continuously on its  endpoints.  

In case (b), 
$$\varpi^{\check\kappa} 
 \le\varpi^{\kappa}, \text{ \ \ where } \check\kappa =\frac{\kappa_F + \epsilon}{f ( p )^2}.$$
Therefore if $|(p,\varphi)\,(\psi,p)|_{\{p\}\times F}<\varpi^{\check\kappa} $, then 
$(p,\varphi), (p,\psi)$ are joined by a geodesic $\gamma$  of length  $<\varpi^{\kappa}$ in $B \times_fF$. Since $B \times_fF \in\CAT^{\,\kappa}$, then $\gamma$  depends uniquely and continuously on its  endpoints.  

Now the claim follows from Lemma \ref{lem:fibergeo}.

\begin{claim}
$F\in\CAT^{\,\kappa_F}$  where $\kappa_F$ is defined in (i) and (ii).
\end{claim}
The claim follows from claims \ref{F-comparison} and \ref{fiber-geos}, completeness of $F$,  and   Alexandrov's patchwork globalization theorem (see \cite[Proposition II.4.9]{bh} or \cite[\emph{Definitions of CBA}]{akp-book}).
\end{proof}


\begin{thebibliography}{BFK}

\bibitem[AB 96]{ab-cbc} S. Alexander, R.  Bishop,
{\em Comparison Theorems for Curves of
Bounded Geodesic Curvature in Metric Spaces of Curvature Bounded
Above}, Differential Geometry and its Applications {\bf 6} (1996),
67-86.

\bibitem[AB 98]{had-wp}
S.  Alexander, R.  Bishop,
{\em Warped products of Hadamard spaces},
Manuscripta Math. {\bf 96} (1998), 487-505.

\bibitem[AB 04]{ab-wp}
S. Alexander, R. Bishop, {\em Curvature bounds for warped products of
metric spaces}, Geom. Funct. Anal. {\bf14} (2004), 1143-1181.


\bibitem[AB 06]{AB-gauss} 
S.  Alexander, R.  Bishop,
{\em Gauss equation and injectivity radii for subspaces in
spaces of curvature bounded above}, Geom. Dedicata {\bf 117} (2006), 65-84.

\bibitem[AKP]{akp-book}
S. Alexander, V. Kapovitch, A. Petrunin, {\em Alexandrov geometry} (book in preparation). http://www.math.psu.edu/petrunin/papers/alexandrov-geometry/

\bibitem[AKP 08]{akp1}
S. Alexander,  V. Kapovitch, A. Petrunin, 
{\em An optimal lower curvature bound for convex hypersurfaces in Riemannian manifolds}, Illinois J. Math. {\bf 52} (2008), 1031-1033.

\bibitem[BBI 01]{bbi}
D. Burago, Yu. Burago,  S. Ivanov, {\em A Course in Metric Geometry}, Graduate Studies in Mathematics,
Vol. 33. American Mathematical Society, Providence, Rhode Island (2001)

\bibitem[BGP 92]{bgp}
Yu. Burago, M. Gromov, G. Perelman, {\em A. D. Aleksandrov spaces with curvatures bounded below} (Russian), Uspekhi Mat. Nauk {\bf 47} (1992),  3-51; translation in Russian Math. Surveys {\bf 47}  (1992), 1-58.

\bibitem[BH 99]{bh}
M. Bridson, A. Haefliger,  {\em Metric Spaces of Non-positive Curvature}, Springer, Berlin (1999).

\bibitem[Ch 99]{C} C.-H. Chen, {\em Warped products of metric spaces of curvature bounded from above},
Trans. Amer. Math. Soc. {\bf 351} (1999), 4727-4740.

\bibitem[G 53]{g}
C. Goffman, {\em Real Functions}, Rinehart, New York (1953).


\bibitem[LS 97]{lang-schroeder}
U. Lang and V. Schroeder, {\em Kirszbraun's theorem and metric spaces of bounded curvature}, Geom. Funct.
Anal. {\bf 7} (1997),  535-560.

\bibitem[L 04]{lytchak-reach} A. Lytchak, {\em Geometry of sets of
positive reach},  Manuscripta Math. {\bf 115} (2004), 199-205.

\bibitem[Lt 05]{lytchak_open} A. Lytchak, {\em Open Map Theorem for Metric Spaces},  Algebra i Analiz {\bf 17} (2005), 139-159 (Russian); translation in St. Petersburg Math. J. {\bf 17} (2006), 477-491.


\bibitem[Pr 91]{Pm1}
G. Perelman,
{\em Alexandrov's spaces with curvature bounded from below II},
preprint LOMI (1991).

\bibitem[PP 94]{PP}
G. Perelman, A. Petrunin,
{\em Quasigeodesics and Gradient Curves in Alexandrov Spaces},
preprint (1994), www.math.psu.edu/petrunin.

\bibitem[Pt 06]{petrunin_survey}
A. Petrunin, {\em Semiconcave functions in Alexandrov's geometry}, Surveys in differential geometry. Vol. XI, 137Ð201,
Surv. Differ. Geom., 11, Int. Press, Somerville, MA   (2007). 

\bibitem[Pt 12]{globalization}
A. Petrunin, {\em A globalization for non-complete but geodesic spaces}, 
to appear in Math. Annalen. 
arXiv:1208.3155v3 [math.DG] .

\bibitem[Pl 02]{pl} C. Plaut, {\em Metric spaces of curvature $\ge k$}.
R. Daverman, R. Sher (eds.), {\em Handbook of Geometric Topology},
Elsevier, (2002), 819-898.

\bibitem[S 93]{shio} Shiohama, {\em An introduction to the geometry of Alexandrov spaces}. Lecture Notes Series, 8. Seoul National University, Research Institute of Mathematics, Global Analysis Research Center, Seoul (1993).

\end{thebibliography}
\end{document}